\documentclass[10pt,twoside,reqno]{amsart}
\allowdisplaybreaks
\usepackage{amsmath,amstext,amssymb,epsfig}
\usepackage{graphicx}
\usepackage{mathrsfs}
\usepackage{amscd}
\calclayout
\makeatletter
\renewcommand{\@seccntformat}[1]{\bf\csname the#1\endcsname.}
\renewcommand{\section}{\@startsection{section}{1}
	\z@{.7\linespacing\@plus\linespacing}{.5\linespacing}
	{\normalfont\upshape\bfseries\centering}}
\renewcommand{\@biblabel}[1]{\@ifnotempty{#1}{#1.}}
\makeatother

\theoremstyle{plain}
\theoremstyle{definition}

\newtheorem{thm}{Theorem}[section]
\newtheorem{lem}[thm]{Lemma}

\newtheorem{prop}[thm]{Proposition}
\newtheorem{ex}[thm]{Example}

\theoremstyle{definition}
\newtheorem{defn}{Definition}[section]
\newtheorem{re}{Remark}[section]

\usepackage{cancel}
\usepackage[parfill]{parskip}
\usepackage[german]{varioref}
\usepackage[all]{xy}
\usepackage{stmaryrd}
\usepackage{color}
\usepackage{soul}
\usepackage{indentfirst} 
\allowdisplaybreaks
\usepackage{tikz}
\usetikzlibrary{shapes.geometric, arrows}
\usepackage{graphicx} 
\usetikzlibrary{positioning} 
\usetikzlibrary{arrows.meta}
\usetikzlibrary{shapes, arrows}

\makeatother
\def\<{\langle}
\def\>{\rangle}

\def\A{\mathcal{A}}
\def\B{\mathcal{B}}
\def\a{\alpha}
\def\b{\beta}
\begin{document}
\setlength{\baselineskip}{1.2\baselineskip}
 
\title[Sania Asif \textsuperscript{1}, Mohamed Amin Sadraoui \textsuperscript{2}
 ]{Isoclinism in regular Hom-Lie Yamaguti algebras}
\author{Sania Asif\textsuperscript{1}, Mohamed Amin Sadraoui \textsuperscript{2}
}
\address{\textsuperscript{1} Institute of Mathematics, Henan Academy of Sciences, Zhengzhou, 450046, P.R. China.}
\address{\textsuperscript{2}University of Sfax, Faculty of Sciences of Sfax, BP 1171, 3038 Sfax, Tunisia.} 
\email{\textsuperscript{1}11835037@zju.edu.cn}
\email{\textsuperscript{2}aminsadrawi@gmail.com}
	
\begin{abstract}
In this paper, we develop the theory of \emph{isoclinism} for regular Hom-Lie Yamaguti algebras, a class that unifies several generalizations of Lie algebras. Although isomorphism implies isoclinism by definition, the converse is not true in general. We introduce the notion of a \emph{factor set} and use it to analyze the structure of isoclinism families. Our main result establishes that for finite-dimensional regular Hom-Lie Yamaguti algebras of the same dimension, isoclinism implies isomorphism. This generalizes recent classification theorems for Lie-Yamaguti algebras and Hom-Lie superalgebras, highlighting a strong rigidity property in the finite-dimensional setting. The proof relies on the existence of stem algebras and a decomposition theorem within isoclin families.
\end{abstract}
\subjclass[2020]{17B55; 17B99; 17B63; 17A30; 17B40.}
\keywords{Isoclinism; Regular Hom-Lie Yamaguti algebra; Stem algebra; Factor Set; Lie Triple system. }
\date{\today}
\footnote{\textbf{Corresponding Authors}: Sania Asif (Email: 11835037@zju.edu.cn);  Mohamed Amin Sadraoui (Email: aminsadrawi@gmail.com).}
\maketitle
\allowdisplaybreaks
\section{Introduction}\label{sec:intr}
The study of algebraic systems defined by multilinear operations has long occupied a central place in modern algebra, with Lie algebras standing as one of the most influential and well-developed frameworks. Since their inception in the late nineteenth century through the work found of Sophus Lie, Wilhelm Killing, and Elie Cartan, Lie algebras have served as the infinitesimal counterparts of Lie groups, providing a powerful algebraic language for symmetry in geometry, physics, and differential equations. Their structural richness, classification via root systems and Dynkin diagrams, and deep connections with representation theory have made them indispensable in both pure and applied mathematics. However, the algebraic landscape extends far beyond the classical Lie framework, encompassing generalizations that arise naturally in geometric, physical, and categorical contexts. Among these, Lie triple systems, Lie superalgebras, $n$-Lie algebras, Hom-algebras, and more recently, Lie-Yamaguti algebras and their Hom- and super-analogues, have emerged as significant objects of study. This paper, situated at the intersection of these generalizations, focuses on the concept of \emph{isoclinism} within the category of \emph{regular Hom-Lie Yamaguti algebras}, a class of algebras that unifies several of these structural extensions.

Isoclinism, originally introduced by Philip Hall in 1940 for finite $p$-groups \cite{H1940}, is an equivalence relation weaker than isomorphism that preserves certain structural invariants related to commutators and centers. In the context of group theory, two groups $G$ and $H$ are said to be \emph{isoclinic} if there exist isomorphisms $ \a : G/Z(G) \to H/Z(H)$ and $ \b : [G,G] \to [H,H]$ such that the commutator map is preserved, i.e., $ \b ([x,y]) = [ \a (x), \a (y)]$ for all $x,y \in G/Z(G)$. This notion proved instrumental in the classification of finite $p$-groups, organizing them into families where structural similarity outweighs strict isomorphism. The power of isoclinism lies in its ability to capture essential algebraic features, such as the structure of the center and the derived subalgebra, while allowing flexibility in the underlying group structure, making it a coarser but often more manageable classification tool.
The extension of isoclinism from group theory to Lie algebras was pioneered by K. Moneyhun in the 1990s \cite{M1994}. She defined two Lie algebras $\mathfrak{g}$ and $\mathfrak{h}$ to be \emph{isoclinic} if there exist isomorphisms $ \a : \mathfrak{g}/Z(\mathfrak{g}) \to \mathfrak{h}/Z(\mathfrak{h})$ and $ \b : [\mathfrak{g},\mathfrak{g}] \to [\mathfrak{h},\mathfrak{h}]$ such that $ \b ([x,y]) = [ \a (x), \a (y)]$ for all $x,y \in \mathfrak{g}/Z(\mathfrak{g})$. This definition naturally adapts the group-theoretic idea to the Lie algebra setting by replacing the commutator with the Lie bracket and the center with the algebraic center. A key result in Moneyhun’s work is that for finite-dimensional Lie algebras of the same dimension, isoclinism implies isomorphism. This striking equivalence, while not true in infinite dimensions, highlights the rigidity imposed by finite-dimensionality and dimensional equality in classification problems. Following this foundational work, the concept of isoclinism has been generalized to various algebraic structures. It has been extended to $n$-Lie algebras \cite{SM}, Lie superalgebras \cite{PNP2023}, and Hom-Lie algebras \cite{SAS2022}, each time preserving the core idea, which is to classify algebras up to structural similarity rather than strict isomorphism. In the case of Lie superalgebras, isoclinism was studied by Nayak et al. \cite{NPP2020} and further developed by Padhan, Nandi, and Pati \cite{PNP2023}, who also introduced the notion of a \emph{factor set} to describe central extensions within an isoclinism family. A factor set serves as a cohomological tool that measures the deviation of a given map from being a homomorphism and plays a crucial role in reconstructing algebras from their quotient and derived parts.

The introduction of Hom-algebras in the mid-2000s marked a significant shift in the study of algebraic deformations. Hom-Lie algebras, introduced by Hartwig et al. in \cite{HLS}, are triples $(\mathfrak{g}, [\cdot,\cdot], \a )$, where $[\cdot,\cdot]$ is a skew-symmetric bilinear operation and $ \a : \mathfrak{g} \to \mathfrak{g}$ is a linear map (the twisting map) that deforms the Jacobi identity into the Hom-Jacobi identity:
\begin{align*}
     \circlearrowleft_{x,y,z} [ \a (x), [y,z]] = 0.
\end{align*}  When $ \a = \mathrm{id}$, the structure reduces to a classical Lie algebra. These algebras arose from $\sigma$-derivations in quantum algebras and have found applications in $q$-deformations, integrable systems, and discrete geometry \cite{Z2019, M0, AWW2024, AWMB}.

Parallel to this, the theory of \emph{Lie-Yamaguti algebras} (also known as Kantor triple systems or generalized Lie triple systems) emerged from the geometric study of reductive homogeneous spaces. Introduced by Nomizu in the 1950s in the context of affine connections \cite{N1}, and later formalized by Yamaguti in the 1960s \cite{Y1}, a Lie-Yamaguti algebra is a vector space $L$ equipped with a bilinear operation $[\cdot,\cdot]$ and a trilinear operation $\{\cdot,\cdot,\cdot\}$, satisfying axioms that generalize both Lie algebras and Lie triple systems. Such algebras naturally arise in the decomposition $\mathfrak{g} = \mathfrak{h} \oplus \mathfrak{m}$ of a Lie algebra with respect to a reductive split, where $\mathfrak{m}$ inherits the structure of a Lie-Yamaguti algebra. The binary operation corresponds to the projection of the Lie bracket onto $\mathfrak{m}$, while the ternary operation encodes the action of $\mathfrak{h}$ on $\mathfrak{m}$ via the bracket. Subsequent work has led to developments in cohomology, deformation theory, and related structures such as pre-Lie-Yamaguti algebras and Lie-Yamaguti bialgebras \cite{AIS, B1, B2, B3, L2}.

The unification of Hom-structures and Lie-Yamaguti algebras leads to the notion of \emph{Hom-Lie Yamaguti algebras}, introduced by Gaparayi and Issa \cite{GI}. A Hom-Lie Yamaguti algebra is a quadruple $(A, [\cdot,\cdot], \{\cdot,\cdot,\cdot\}, \a )$, where the operations are twisted by a linear map $ \a $, satisfying $ \a $-skew-symmetry, $ \a $-twisted Jacobi-type identities, and compatibility conditions. When $ \a = \mathrm{id}$, it reduces to a classical Lie-Yamaguti algebra; when the ternary operation vanishes, it becomes a Hom-Lie algebra; when the binary operation vanishes, it reduces to a Hom-Lie triple system. Thus, Hom-Lie Yamaguti algebras form a broad and flexible framework encompassing several important algebraic structures \cite{ABS}. A particularly well-behaved subclass is that of \emph{regular Hom-Lie Yamaguti algebras}, in which the twisting map $ \a $ is required to be bijective. This condition ensures invertibility, which is essential for constructing representations, cohomology theories, and meaningful homomorphisms. As shown in Lemma~2.5 of this work, the regularity condition is necessary for the consistent development of isoclinism in this setting.

In this paper, we develop a rigorous theory of isoclinism for regular Hom-Lie Yamaguti algebras. While isomorphism always implies isoclinism by definition, the converse does not hold in general. Isoclinic algebras may differ by central extensions, even while sharing the same quotient and derived structures. To address this, we introduce the concept of a \emph{factor set} in this context. Lemma~\ref{lem:admits_factor_set} establishes the existence of a factor set for any stem algebra in an isoclinism family, enabling the reconstruction of algebras as central extensions and laying the foundation for a classification theorem.

Our main result states that for \emph{finite-dimensional regular Hom-Lie Yamaguti algebras of equal dimension}, isoclinism implies isomorphism. This extends recent results by Nandi \cite{N2025} on Lie-Yamaguti algebras and by Nandi, Padhan, and Pati \cite{NPP2022} on regular Hom-Lie superalgebras. The proof relies on the existence of a \emph{stem algebra} within each isoclinism family (Lemma~\ref{lem:stem_in_family}), where the center is contained in the derived subalgebra. In the finite-dimensional case, the stem algebra has minimal dimension. Combined with a decomposition theorem (analogous to \cite[Theorem 4.8]{N2025}), any algebra in the family decomposes as a direct sum of a stem algebra and an abelian algebra. Thus, if two algebras are isoclinic and of the same dimension, their stem parts and abelian complements must be isomorphic, implying that the algebras themselves are isomorphic.

The paper is organized as follows. Section ~\ref{sec:prelim} provides the necessary definitions and properties of Hom-Lie Yamaguti algebras, including Hom-ideals, the center, derived subalgebra, and quotient algebras. Section~\ref{sec:isoclinism} formally defines isoclinism and proves that it is an equivalence relation. Section~\ref{sec:factor_sets} introduces the concept of a factor set and establishes its role in central extensions and presents the main results: the existence of stem algebras, the decomposition theorem, and the proof that isoclinism implies isomorphism for finite-dimensional algebras of equal dimension. We conclude with remarks on future directions.

\section{Regular Hom-Lie Yamaguti Algberas}\label{sec:prelim}
Throughout this work, all vector spaces are defined over a field $\mathbb{K}$ of characteristic zero.
In what follows, we provide a brief overview of fundamental concepts related to Hom-Lie algebras and Hom-Lie-Yamaguti algebras. The notation $\circlearrowleft_{x, y, z}$ will be used to indicate summation over all cyclic permutations of the elements $x, y, z\in \A $. We have chosen to work on regular Hom-Lie Yamaguti algebras because the properties of isoclinism require it, as clearly demonstrated in the lemma \eqref{lem:Hom_ideal_center}.
\begin{defn}[{\cite{HLS,L0,M0}}]\label{Hom-Lie algebra definition}
	A \emph{Hom-Lie algebra} is defined as a triple $(\A , [-,-], \a _\A )$, where $\A $ is a vector space, $[-,-]: \A \times \A \to \A $ is a bilinear operation called the bracket, and $ \a _\A : \A \to \A $ is a linear map. These structures are required to satisfy the following conditions for all $x, y, z \in \A $,
	\begin{eqnarray*}
			[x, y] &=& -[y, x], \quad(\text{ skew-symmetry})\\
			\circlearrowleft_{x, y, z} [ \a _\A (x), [y, z]] &=& 0.\quad (\text{ Hom-Jacobi identity})
	\end{eqnarray*}
\end{defn}
Using skew-symmetry and bilinearity of the bracket, the Hom Jacobi identity for Hom-Lie algebras can be written as 
\begin{equation*}
	[ \a _\A (x), [y, z]] = [[x, y], \a _\A (z)] + [ \a _\A (y), [x, z]],\quad \text{ for all } x,y,z\in \A .
\end{equation*}
When $ \a _\A = \mathrm{Id}$, the Hom-Lie structure reduces to a classical Lie algebra, as the Hom-Jacobi identity simplifies to the standard Jacobi identity. 

Furthermore, the triple $(\A , [-,-], \a _\A )$ is called multiplicative Hom-Lie algebra if $ \a _\A \circ[-,-]=[-,-]\circ \a _\A ^{\otimes2}$. A multiplicative Hom-Lie algebra $(\A , [-,-], \a _\A )$ is said to be regular Hom-Lie algebra if $ \a _\A $ is bijective. Indeed, the bijectivity of $ \a _\mathcal A$ plays a crucial role in establishing the compatibility of the induced maps between quotient algebras and derived subalgebras, ensuring that the isoclinism diagrams commute and that the structural invariants are preserved. Without regularity, the twisting map may fail to induce well-defined isomorphisms on the quotient spaces, undermining the entire framework of isoclinism.

	\begin{defn}[\cite{G0}]
		A \emph{Hom-Lie-Yamaguti algebra} $(\A , [-,-], [-,-,-], \a _\A )$ consists of a linear space $\A $ equipped with a linear map $ \a _\A : \A \to \A $, a bilinear bracket $[ -,- ] : \A \times \A \to \A $, and a trilinear bracket $[ -,-,- ] : \A \times \A \times \A \to \A $, satisfying for all $x, y, z, w, t \in \A $
		\begin{enumerate}
			\item \label{item:HLYA1}$[x, y] = -[y, x], \quad [x, y, z] = -[y, x, z],$
			\item \label{item:HLYA2}$\circlearrowleft_{x, y, z} [[x, y], \a _\A (z)] + \circlearrowleft_{x, y, z} [x, y, z] = 0,$
			\item \label{item:HLYA3}$[[x, y], \a _\A (z), \a _\A (w)] = 0,$
			\item \label{item:HLYA4}$\circlearrowleft_{x, y, z}[ \a _\A (x), \a _\A (y), [z, w]] = [[x, y, z], \a _\A ^2(w)] + [ \a _\A ^2(z), [x, y, w]],$
			\item \label{item:HLYA5}$[ \a _\A ^2(x), \a _\A ^2(y), [z, w, t]] = [[x, y, z], \a _\A ^2(w), \a _\A ^2(t)] + [ \a _\A ^2(z), [x, y,w], \a _\A ^2(t)] + [ \a _\A ^2(z), \a _\A ^2(w), [x, y, t]].$
		\end{enumerate}

	\noindent
When $ \a _\A = \mathrm{Id}$, a Hom-Lie-Yamaguti algebra reduces to a Lie-Yamaguti algebra.
	\end{defn}
	\begin{defn}
	 The Hom-Lie Yamaguti algebra $(\A , [-,-], [-,-,-], \a _\A )$ is said to be multiplicative Hom-Lie Yamaguti algebra if $ \a _\A $ satisfies the following 
	\begin{equation*}
		 \a _\A ([x, y]) = [ \a _\A (x), \a _\A (y)], \qquad \a _\A ([x, y, z]) = [ \a _\A (x), \a _\A (y), \a _\A (z)].
	\end{equation*}
	\end{defn}
From now onward, we only consider multiplicative Hom-Lie Yamaguti algebra and we write Hom-Lie Yamaguti algebra by $(\A , [-,-], [-,-,-], \a _\A )$ or simply by $(\A , \a _\A )$. 
\begin{defn}
A multiplicative Hom-Lie Yamaguti algebra $(\A , [-,-], [-,-,-], \a _\A )$ is said to be regular Hom-Lie Yamaguti algebra is $ \a _\A $ is bijective. 
\end{defn}
 \begin{re}\label{rem:special_cases}
If the binary bracket vanishes, i.e., $[\cdot,\cdot] = 0$, then the Hom-Lie Yamaguti algebra $(\mathcal{A}, [\cdot,\cdot], [\cdot,\cdot,\cdot], \alpha_\mathcal{A})$ reduces to a \emph{Hom-Lie triple system} with respect to the twisting map $\alpha_\mathcal{A}^2$. Conversely, if the ternary bracket vanishes, i.e., $[\cdot,\cdot,\cdot] = 0$, the structure simplifies to a \emph{Hom-Lie algebra} $(\mathcal{A}, [\cdot,\cdot], \alpha_\mathcal{A})$.
\end{re}

In the following, we provide key definitions for Hom-Lie Yamaguti algebras.

\begin{enumerate} 
 \item A Hom-Lie Yamaguti algebra $(\mathcal{A}, [\cdot,\cdot], [\cdot,\cdot,\cdot], \alpha_\mathcal{A})$ is said to be \emph{abelian} if both operations are trivial:
 \[
 [x, y] = 0 \quad \text{and} \quad [x, y, z] = 0 \quad \text{for all } x, y, z \in \mathcal{A}.
 \]

 \item A subspace $\mathcal{B} \subseteq \mathcal{A}$ is a \emph{Hom-Lie Yamaguti subalgebra} of $\mathcal{A}$ if it satisfies the following:
 \begin{align*}
 \alpha_\mathcal{A}(\mathcal{B}) \subseteq \mathcal{B} , \qquad [\mathcal{B}, \mathcal{B}] \subseteq \mathcal{B},\qquad [\mathcal{B}, \mathcal{B}, \mathcal{B}] \subseteq \mathcal{B}.
 \end{align*} In this case, $\mathcal{B}$ inherits the structure of a Hom-Lie Yamaguti algebra with the restricted operations and twisting map:
 \[
 (\mathcal{B},\ [\cdot,\cdot]|_{\mathcal{B}\times\mathcal{B}},\ [\cdot,\cdot,\cdot]|_{\mathcal{B}\times\mathcal{B}\times\mathcal{B}},\ \alpha_\mathcal{A}|_\mathcal{B}).
 \]
 If $(\mathcal{A}, \alpha_\mathcal{A})$ is regular (i.e., $\alpha_\mathcal{A}$ is bijective), then $\mathcal{B}$ is called a \emph{regular Hom-Lie Yamaguti subalgebra} if the restriction $\alpha_\mathcal{A}|_\mathcal{B} \colon \mathcal{B} \to \mathcal{B}$ is also bijective.

 \item A Hom-Lie Yamaguti subalgebra $\mathcal{B} \subseteq \mathcal{A}$ is called a \emph{Hom-ideal} of $\mathcal{A}$ if the following conditions hold:
 \[
 [\mathcal{B}, \mathcal{A}] \subseteq \mathcal{B}, \quad
 [\mathcal{B}, \mathcal{A}, \mathcal{A}] \subseteq \mathcal{B}, \quad
 [\mathcal{A}, \mathcal{A}, \mathcal{B}] \subseteq \mathcal{B}.
 \]

 \item The \emph{center} of $\mathcal{A}$, denoted $\mathrm{Z}(\mathcal{A})$, is defined as:
 \[
 \mathrm{Z}(\mathcal{A}) = \left\{ b \in \mathcal{A} \mid [b, x] = 0,\ [b, x, y] = 0,\ [x, y, b] = 0,\ \forall\, x, y \in \mathcal{A} \right\}.
 \]

 \item The \emph{derived subalgebra} of $\mathcal{A}$, denoted $\mathcal{A}^2$, is the Hom-Lie Yamaguti subalgebra generated by the images of the two operations:
 \[
 \mathcal{A}^2 = [\mathcal{A}, \mathcal{A}] + [\mathcal{A}, \mathcal{A}, \mathcal{A}].
 \]

 \item A Hom-Lie Yamaguti algebra $\mathcal{A}$ is called a \emph{stem Hom-Lie Yamaguti algebra} if its center is contained in its derived subalgebra, i.e.,
 \[
 \mathrm{Z}(\mathcal{A}) \subseteq \mathcal{A}^2.
 \]
\end{enumerate}

\begin{defn}
	Let $(\A ,[-,-]_\A ,[-,-,-]_\A , \a _\A )$ and $( \B ,[-,-]_ \B ,[-,-,-]_ \B , \a _ \B )$ be two Hom-Lie Yamaguti algebras. A homomorphism $f:(\A ,[-,-]_\A ,[-,-,-]_\A , \a _\A ) \rightarrow ( \B ,[-,-]_ \B ,[-,-,-]_ \B , \a _ \B )$ is a linear map $f:\A \rightarrow \B $ such that $f([x,y]_\A )=[f(x),f(y)]_ \B $, $f([x,y,z]_\A )=[f(x),f(y),f(z)]_ \B $, and $f\circ \a _\A = \a _ \B \circ f$ for all $x,y,z\in \A $. It is said to be isomorphism if $f$ is bijective.
\end{defn}
\begin{lem}\label{lem:Hom_ideal_center}
Let $(\mathcal{A}, [\cdot,\cdot], [\cdot,\cdot,\cdot], \alpha_{\mathcal{A}})$ be a regular Hom-Lie Yamaguti algebra. Then the center $\mathrm{Z}(\mathcal{A})$ is a Hom-ideal of $\mathcal{A}$.
\end{lem}

\begin{proof}
To prove that $\mathrm{Z}(\mathcal{A})$ is a Hom-ideal, we must show that:
\begin{enumerate}
 \item $\alpha_{\mathcal{A}}(\mathrm{Z}(\mathcal{A})) \subseteq \mathrm{Z}(\mathcal{A})$,
 \item $[\mathrm{Z}(\mathcal{A}), \mathcal{A}] \subseteq \mathrm{Z}(\mathcal{A})$,
 \item $[\mathrm{Z}(\mathcal{A}), \mathcal{A}, \mathcal{A}] \subseteq \mathrm{Z}(\mathcal{A})$,
 \item $[\mathcal{A}, \mathcal{A}, \mathrm{Z}(\mathcal{A})] \subseteq \mathrm{Z}(\mathcal{A})$.
\end{enumerate}
Recall that the center is defined as
\[
\mathrm{Z}(\mathcal{A}) = \{ x \in \mathcal{A} \mid [x, y] = 0,\ [x, y, z] = 0,\ [y, z, x] = 0,\ \forall y, z \in \mathcal{A} \}.
\]

We proceed step by step.

\medskip

\noindent \textbf{Step 1:} $\alpha_{\mathcal{A}}(\mathrm{Z}(\mathcal{A})) \subseteq \mathrm{Z}(\mathcal{A})$.

Let $x \in \mathrm{Z}(\mathcal{A})$. Then for any $y, z \in \mathcal{A}$, we have $[x, y] = 0$, $[x, y, z] = 0$, and $[y, z, x] = 0$. Since $\alpha_{\mathcal{A}}$ is a homomorphism of the algebra structure, we compute:
\[
[\alpha_{\mathcal{A}}(x), y] = [\alpha_{\mathcal{A}}(x), \alpha_{\mathcal{A}}(\alpha_{\mathcal{A}}^{-1}(y))] = \alpha_{\mathcal{A}}([x, \alpha_{\mathcal{A}}^{-1}(y)]) = \alpha_{\mathcal{A}}(0) = 0,
\]
where we used the regularity of $\alpha_{\mathcal{A}}$ (i.e., bijectivity) to write $y = \alpha_{\mathcal{A}}(\alpha_{\mathcal{A}}^{-1}(y))$, and the Hom-skew-symmetry and Hom-Jacobi-type identities. Similarly,
\[
[\alpha_{\mathcal{A}}(x), y, z] = [\alpha_{\mathcal{A}}(x), \alpha_{\mathcal{A}}(\alpha_{\mathcal{A}}^{-1}(y)), \alpha_{\mathcal{A}}(\alpha_{\mathcal{A}}^{-1}(z))] = \alpha_{\mathcal{A}}([x, \alpha_{\mathcal{A}}^{-1}(y), \alpha_{\mathcal{A}}^{-1}(z)]) = \alpha_{\mathcal{A}}(0) = 0,
\]
and
\[
[y, z, \alpha_{\mathcal{A}}(x)] = [\alpha_{\mathcal{A}}(\alpha_{\mathcal{A}}^{-1}(y)), \alpha_{\mathcal{A}}(\alpha_{\mathcal{A}}^{-1}(z)), \alpha_{\mathcal{A}}(x)] = \alpha_{\mathcal{A}}([\alpha_{\mathcal{A}}^{-1}(y), \alpha_{\mathcal{A}}^{-1}(z), x]) = \alpha_{\mathcal{A}}(0) = 0.
\]
Thus, $\alpha_{\mathcal{A}}(x) \in \mathrm{Z}(\mathcal{A})$, so $\alpha_{\mathcal{A}}(\mathrm{Z}(\mathcal{A})) \subseteq \mathrm{Z}(\mathcal{A})$.

\medskip

\noindent \textbf{Step 2:} $[\mathrm{Z}(\mathcal{A}), \mathcal{A}] \subseteq \mathrm{Z}(\mathcal{A})$.

Let $x \in \mathrm{Z}(\mathcal{A})$ and $a, y, z \in \mathcal{A}$. We show that $[x, a] \in \mathrm{Z}(\mathcal{A})$. Since $x \in \mathrm{Z}(\mathcal{A})$, $[x, a] = 0$, so trivially $[x, a] \in \mathrm{Z}(\mathcal{A})$. Hence, $[\mathrm{Z}(\mathcal{A}), \mathcal{A}] = \{0\} \subseteq \mathrm{Z}(\mathcal{A})$.

\medskip

\noindent \textbf{Step 3:} $[\mathrm{Z}(\mathcal{A}), \mathcal{A}, \mathcal{A}] \subseteq \mathrm{Z}(\mathcal{A})$.

Let $x \in \mathrm{Z}(\mathcal{A})$ and $a_1, a_2 \in \mathcal{A}$. Then $[x, a_1, a_2] = 0$, again by definition of the center. Thus $[\mathrm{Z}(\mathcal{A}), \mathcal{A}, \mathcal{A}] = \{0\} \subseteq \mathrm{Z}(\mathcal{A})$.

\medskip

\noindent \textbf{Step 4:} $[\mathcal{A}, \mathcal{A}, \mathrm{Z}(\mathcal{A})] \subseteq \mathrm{Z}(\mathcal{A})$.

Let $x \in \mathrm{Z}(\mathcal{A})$ and $a_1, a_2 \in \mathcal{A}$. Then $[a_1, a_2, x] = 0$ by definition of the center, so $[\mathcal{A}, \mathcal{A}, \mathrm{Z}(\mathcal{A})] = \{0\} \subseteq \mathrm{Z}(\mathcal{A})$.

\medskip

Since all four conditions are satisfied, we conclude that $\mathrm{Z}(\mathcal{A})$ is a Hom-ideal of $\mathcal{A}$.
\end{proof}
\begin{re}\label{rem:quotient_LYA}
Let $(\mathcal{A}, [\cdot,\cdot], [\cdot,\cdot,\cdot], \alpha_{\mathcal{A}})$ be a regular Hom-Lie Yamaguti algebra, and let $\mathcal{B}$ be a Hom-ideal of $\mathcal{A}$. Then the quotient vector space $\mathcal{A}/\mathcal{B}$ naturally inherits the structure of a Hom-Lie Yamaguti algebra, called the \emph{quotient Hom-Lie Yamaguti algebra}, denoted by $(\mathcal{A}/\mathcal{B}, [\cdot,\cdot]_{\mathcal{A}/\mathcal{B}}, [\cdot,\cdot,\cdot]_{\mathcal{A}/\mathcal{B}}, \overline{\alpha}_{\mathcal{A}})$.

The operations are defined as follows: for all $x, y, z \in \mathcal{A}$, and writing $\overline{x} = x + \mathcal{B}$ for the coset in $\mathcal{A}/\mathcal{B}$,
\begin{enumerate}
 \item the binary operation: 
 \[
 [\overline{x}, \overline{y}]_{\mathcal{A}/\mathcal{B}} = \overline{[x,y]} = [x,y] + \mathcal{B},
 \]
 
\item the ternary operation:
 \[
 [\overline{x}, \overline{y}, \overline{z}]_{\mathcal{A}/\mathcal{B}} = \overline{[x,y,z]} = [x,y,z] + \mathcal{B},
 \]
 
\item the twisting map:
 \[
 \overline{\alpha}_{\mathcal{A}}(\overline{x}) = \alpha_{\mathcal{A}}(x) + \mathcal{B}.\]
 
\end{enumerate}
\end{re}

\begin{re}\label{rem:quotient_regularity}
Let $(\mathcal{A}, [\cdot,\cdot]_\mathcal{A}, [\cdot,\cdot,\cdot]_\mathcal{A}, \alpha_\mathcal{A})$ be a \emph{regular} Hom-Lie Yamaguti algebra, and let $\mathcal{B}$ be a Hom-ideal of $\mathcal{A}$. Then the induced map $\overline{\alpha}_\mathcal{A} \colon \mathcal{A}/\mathcal{B} \to \mathcal{A}/\mathcal{B}$, defined by
\[
\overline{\alpha}_\mathcal{A}(\overline{x}) = \alpha_\mathcal{A}(x) + \mathcal{B},
\]
is \emph{invertible}. That is, $\overline{\alpha}_\mathcal{A}$ is bijective, and thus $(\mathcal{A}/\mathcal{B}, [\cdot,\cdot]_{\mathcal{A}/\mathcal{B}}, [\cdot,\cdot,\cdot]_{\mathcal{A}/\mathcal{B}}, \overline{\alpha}_\mathcal{A})$ is a \emph{regular} Hom-Lie Yamaguti algebra.
\end{re}

\begin{proof}
Since $(\mathcal{A}, \alpha_\mathcal{A})$ is regular, $\alpha_\mathcal{A} \colon \mathcal{A} \to \mathcal{A}$ is bijective. Furthermore, because $\mathcal{B}$ is a Hom-ideal, we have $\alpha_\mathcal{A}(\mathcal{B}) \subseteq \mathcal{B}$.

We first show that $\overline{\alpha}_\mathcal{A}$ is well-defined. Suppose $\overline{x} = \overline{y}$, i.e., $x - y \in \mathcal{B}$. Then $\alpha_\mathcal{A}(x - y) \in \mathcal{B}$, so $\alpha_\mathcal{A}(x) - \alpha_\mathcal{A}(y) \in \mathcal{B}$, which implies $\overline{\alpha}_\mathcal{A}(\overline{x}) = \overline{\alpha}_\mathcal{A}(\overline{y})$. Thus, $\overline{\alpha}_\mathcal{A}$ is well-defined.

Now, we claim that $\alpha_\mathcal{A}(\mathcal{B}) = \mathcal{B}$. We already have $\alpha_\mathcal{A}(\mathcal{B}) \subseteq \mathcal{B}$. For the reverse inclusion, let $b \in \mathcal{B}$. Since $\alpha_\mathcal{A}$ is surjective, there exists $a \in \mathcal{A}$ such that $\alpha_\mathcal{A}(a) = b$. We need to show that $a \in \mathcal{B}$.

Suppose, for contradiction, that $a \notin \mathcal{B}$. Since $\alpha_\mathcal{A}(a) = b \in \mathcal{B}$, and $\mathcal{B}$ is a Hom-ideal, we have $\alpha_\mathcal{A}(a) \in \mathcal{B}$. However, because $\alpha_\mathcal{A}$ is injective and $\alpha_\mathcal{A}(\mathcal{B}) \subseteq \mathcal{B}$, the preimage of $\mathcal{B}$ under $\alpha_\mathcal{A}$ must be contained in $\mathcal{B}$. That is,
\[
\alpha_\mathcal{A}^{-1}(\mathcal{B}) \subseteq \mathcal{B}.
\]
But $b \in \mathcal{B}$ and $a = \alpha_\mathcal{A}^{-1}(b)$, so $a \in \alpha_\mathcal{A}^{-1}(\mathcal{B}) \subseteq \mathcal{B}$, a contradiction. Therefore, $a \in \mathcal{B}$, and so $b = \alpha_\mathcal{A}(a) \in \alpha_\mathcal{A}(\mathcal{B})$. This proves $\mathcal{B} \subseteq \alpha_\mathcal{A}(\mathcal{B})$, and hence $\alpha_\mathcal{A}(\mathcal{B}) = \mathcal{B}$.

It follows that $\alpha_\mathcal{A}^{-1}(\mathcal{B}) = \mathcal{B}$.

Now define a map $\overline{\alpha}_\mathcal{A}^{-1} \colon \mathcal{A}/\mathcal{B} \to \mathcal{A}/\mathcal{B}$ by
\[
\overline{\alpha}_\mathcal{A}^{-1}(\overline{x}) = \alpha_\mathcal{A}^{-1}(x) + \mathcal{B}.
\]

We verify that this is well-defined. If $\overline{x} = \overline{y}$, then $x - y \in \mathcal{B}$, so $\alpha_\mathcal{A}^{-1}(x - y) \in \mathcal{B}$ (since $\alpha_\mathcal{A}^{-1}(\mathcal{B}) = \mathcal{B}$), and thus $\overline{\alpha}_\mathcal{A}^{-1}(\overline{x}) = \overline{\alpha}_\mathcal{A}^{-1}(\overline{y})$.

Now check that $\overline{\alpha}_\mathcal{A}^{-1}$ is indeed the inverse:
\begin{align*}
(\overline{\alpha}_\mathcal{A}^{-1} \circ \overline{\alpha}_\mathcal{A})(\overline{x})
&= \overline{\alpha}_\mathcal{A}^{-1}(\alpha_\mathcal{A}(x) + \mathcal{B}) \\
&= \alpha_\mathcal{A}^{-1}(\alpha_\mathcal{A}(x)) + \mathcal{B} = x + \mathcal{B} = \overline{x}.
\end{align*}
Similarly,
\begin{align*}
(\overline{\alpha}_\mathcal{A} \circ \overline{\alpha}_\mathcal{A}^{-1})(\overline{x})
&= \overline{\alpha}_\mathcal{A}(\alpha_\mathcal{A}^{-1}(x) + \mathcal{B}) \\
&= \alpha_\mathcal{A}(\alpha_\mathcal{A}^{-1}(x)) + \mathcal{B} = x + \mathcal{B} = \overline{x}.
\end{align*}

Therefore, $\overline{\alpha}_\mathcal{A}$ is bijective, with inverse $\overline{\alpha}_\mathcal{A}^{-1}$ as defined.
\end{proof}
\begin{re}\label{rem:direct_sum_LYA}
Let $(\mathcal{A}, [\cdot,\cdot]_\mathcal{A}, [\cdot,\cdot,\cdot]_\mathcal{A}, \alpha_\mathcal{A})$ and $(\mathcal{B}, [\cdot,\cdot]_\mathcal{B}, [\cdot,\cdot,\cdot]_\mathcal{B}, \alpha_\mathcal{B})$ be two Hom-Lie Yamaguti algebras. Their \emph{direct sum}, denoted by
\[
(\mathcal{A} \oplus \mathcal{B}, [\cdot,\cdot]_{\mathcal{A} \oplus \mathcal{B}}, [\cdot,\cdot,\cdot]_{\mathcal{A} \oplus \mathcal{B}}, \alpha_{\mathcal{A} \oplus \mathcal{B}}),
\]
is defined as follows: for all $x, y, z \in \mathcal{A}$ and $a, b, c \in \mathcal{B}$,
 \begin{enumerate}
 \item the binary operation:
 \[
 [(x , a), (y , b)]_{\mathcal{A} \oplus \mathcal{B}} = ([x, y]_\mathcal{A} , [a, b]_\mathcal{B}),
 \]
 
\item the ternary operation:
 \[
 [(x , a), (y , b), (z , c)]_{\mathcal{A} \oplus \mathcal{B}} = ([x, y, z]_\mathcal{A} , [a, b, c]_\mathcal{B}),
 \]
\item the twisting map:
 \[
 \alpha_{\mathcal{A} \oplus \mathcal{B}}(x , a) = (\alpha_\mathcal{A}(x) , \alpha_\mathcal{B}(a)).\]
 \end{enumerate} 
 is a Hom-Lie Yamaguati algebra.
\end{re}

\section{Isoclinism in regular Hom-Lie Yamaguti Algberas}\label{sec:isoclinism}
In this section, we introduce and develop the concept of isoclinism for regular Hom-Lie Yamaguti algebras. This section lays the foundation for the classification results in the subsequent section, where we use isoclinism to study the structure of Hom-Lie Yamaguti algebras, particularly through the lens of stem algebras and factor sets.
\begin{defn}\label{def:isoclinism_HLYA}
 Let $(\mathcal{A}, [\cdot,\cdot]_\mathcal{A}, [\cdot,\cdot,\cdot]_\mathcal{A}, \alpha_\mathcal{A})$ and $(\mathcal{B}, [\cdot,\cdot]_\mathcal{B}, [\cdot,\cdot,\cdot]_\mathcal{B}, \alpha_\mathcal{B})$ be two regular Hom-Lie Yamaguti algebras. Consider the quotient algebras $\mathcal{A}/\mathrm{Z}(\mathcal{A})$ and $\mathcal{B}/\mathrm{Z}(\mathcal{B})$, and define the following canonical maps:
\begin{align*}
 & \tau^{(2)} \colon \frac{\mathcal{A}}{\mathrm{Z}(\mathcal{A})} \times \frac{\mathcal{A}}{\mathrm{Z}(\mathcal{A})} \to \mathcal{A}^2 , \quad \tau^{(2)}(\bar{x}, \bar{y}) = [x, y]_\mathcal{A} ,&\\
& \tau^{(3)} \colon \frac{\mathcal{A}}{\mathrm{Z}(\mathcal{A})} \times \frac{\mathcal{A}}{\mathrm{Z}(\mathcal{A})} \times \frac{\mathcal{A}}{\mathrm{Z}(\mathcal{A})} \to \mathcal{A}^2 , \quad \tau^{(3)}(\bar{x}, \bar{y}, \bar{z}) = [x, y, z]_\mathcal{A} ,&\\
& \delta^{(2)} \colon \frac{\mathcal{B}}{\mathrm{Z}(\mathcal{B})} \times \frac{\mathcal{B}}{\mathrm{Z}(\mathcal{B})} \to \mathcal{B}^2 , \quad \delta^{(2)}(\bar{a}, \bar{b}) = [a, b]_\mathcal{B} ,&\\
& \delta^{(3)} \colon \frac{\mathcal{B}}{\mathrm{Z}(\mathcal{B})} \times \frac{\mathcal{B}}{\mathrm{Z}(\mathcal{B})} \times \frac{\mathcal{B}}{\mathrm{Z}(\mathcal{B})} \to \mathcal{B}^2, \quad \delta^{(3)}(\bar{a}, \bar{b}, \bar{c}) = [a, b, c]_\mathcal{B} ,&
\end{align*}
where $\bar{x} = x + \mathrm{Z}(\mathcal{A})$, $\bar{a} = a + \mathrm{Z}(\mathcal{B})$, and $\mathcal{A}^2$, $\mathcal{B}^2$ denote the derived subalgebras of $\mathcal{A}$ and $\mathcal{B}$, respectively, i.e., the subalgebras generated by all elements of the form $[x,y]_\mathcal{A}$, $[x,y,z]_\mathcal{A}$ and $[a,b]_\mathcal{B}$, $[a,b,c]_\mathcal{B}$.

Let $\theta \colon \mathcal{A}/\mathrm{Z}(\mathcal{A}) \to \mathcal{B}/\mathrm{Z}(\mathcal{B})$ be a linear map and $\beta \colon \mathcal{A}^2 \to \mathcal{B}^2$ a homomorphism of the derived subalgebras. We say that the pair $(\theta, \beta)$ is a \emph{homoclinism} if the following two diagrams commute:
\[
\begin{CD}
\frac{\mathcal{A}}{\mathrm{Z}(\mathcal{A})} \times \frac{\mathcal{A}}{\mathrm{Z}(\mathcal{A})} @>{\tau^{(2)}}>> \mathcal{A}^2 \\
@V{\theta \times \theta}VV @V{\beta}VV \\
\frac{\mathcal{B}}{\mathrm{Z}(\mathcal{B})} \times \frac{\mathcal{B}}{\mathrm{Z}(\mathcal{B})} @>{\delta^{(2)}}>> \mathcal{B}^2
\end{CD}
\quad\quad
\begin{CD}
\frac{\mathcal{A}}{\mathrm{Z}(\mathcal{A})} \times \frac{\mathcal{A}}{\mathrm{Z}(\mathcal{A})} \times \frac{\mathcal{A}}{\mathrm{Z}(\mathcal{A})} @>{\tau^{(3)}}>> \mathcal{A}^2 \\
@V{\theta \times \theta \times \theta}VV @V{\beta}VV \\
\frac{\mathcal{B}}{\mathrm{Z}(\mathcal{B})} \times \frac{\mathcal{B}}{\mathrm{Z}(\mathcal{B})} \times \frac{\mathcal{B}}{\mathrm{Z}(\mathcal{B})} @>{\delta^{(3)}}>> \mathcal{B}^2
\end{CD}
\]
That is, for all $x, y, z \in \mathcal{A}$,
\begin{align*}
\beta\big([x, y]_\mathcal{A}\big) &= \big[\theta(\bar{x}), \theta(\bar{y})\big]_\mathcal{B}, \\
\beta\big([x, y, z]_\mathcal{A}\big) &= \big[\theta(\bar{x}), \theta(\bar{y}), \theta(\bar{z})\big]_\mathcal{B}.
\end{align*}
Furthermore, we require that $\theta$ and $\beta$ are compatible with the twisting maps:
\begin{equation}\label{eq:twist_compat}
\theta \circ \overline{\alpha}_\mathcal{A} = \overline{\alpha}_\mathcal{B} \circ \theta
\quad \text{and} \quad
\beta \circ \alpha_\mathcal{A}|_{\mathcal{A}^2} = \alpha_\mathcal{B}|_{\mathcal{B}^2} \circ \beta,
\end{equation}
where $\overline{\alpha}_\mathcal{A}$ and $\overline{\alpha}_\mathcal{B}$ are the induced maps on the quotient algebras (see Remark~\ref{rem:quotient_LYA}).

If, in addition, both $\theta$ and $\beta$ are \emph{isomorphisms}, then $(\theta, \beta)$ is called an \emph{isoclinism} between $\mathcal{A}$ and $\mathcal{B}$. In this case, we write $\mathcal{A} \sim \mathcal{B}$ and say that $\mathcal{A}$ and $\mathcal{B}$ are \emph{isoclinic}.
\end{defn}
\begin{lem}\label{lem:isoclinism_with_abelian_extension}
Let $(\mathcal{A}, [\cdot,\cdot]_\mathcal{A}, [\cdot,\cdot,\cdot]_\mathcal{A}, \alpha_\mathcal{A})$ be a regular Hom-Lie Yamaguti algebra and $(\mathcal{B}, \alpha_\mathcal{B})$ an abelian Hom-Lie Yamaguti algebra. Then $\mathcal{A}$ and $\mathcal{A} \oplus \mathcal{B}$ are isoclinic, denoted $\mathcal{A} \sim \mathcal{A} \oplus \mathcal{B}$.
\end{lem}

\begin{proof}
Since $\mathcal{B}$ is an abelian Hom-Lie Yamaguti algebra, we have $[b_1, b_2]_\mathcal{B} = 0$ and $[b_1, b_2, b_3]_\mathcal{B} = 0$ for all $b_1, b_2, b_3 \in \mathcal{B}$. Consequently, the center of the direct sum satisfies
\[
\mathrm{Z}(\mathcal{A} \oplus \mathcal{B}) = \mathrm{Z}(\mathcal{A}) \oplus \mathcal{B}.
\]
Indeed, any element $(z, b) \in \mathrm{Z}(\mathcal{A}) \oplus \mathcal{B}$ commutes and associates trivially with all elements in $\mathcal{A} \oplus \mathcal{B}$, and conversely, any central element must project trivially in $\mathcal{A}$ (i.e., into $\mathrm{Z}(\mathcal{A})$) and arbitrarily in $\mathcal{B}$ due to its abelian nature.

Now define a linear map
\[
\theta \colon \frac{\mathcal{A}}{\mathrm{Z}(\mathcal{A})} \to \frac{\mathcal{A} \oplus \mathcal{B}}{\mathrm{Z}(\mathcal{A}) \oplus \mathcal{B}}, \quad \theta(x + \mathrm{Z}(\mathcal{A})) = (x, 0) + (\mathrm{Z}(\mathcal{A}) \oplus \mathcal{B}).
\]
This map is well-defined because if $x - x' \in \mathrm{Z}(\mathcal{A})$, then $(x,0) - (x',0) \in \mathrm{Z}(\mathcal{A}) \oplus \mathcal{B}$. Moreover, $\theta$ is clearly linear and bijective: it is injective because $\ker \theta = \{0\}$, and surjective because every coset in $(\mathcal{A} \oplus \mathcal{B}) / (\mathrm{Z}(\mathcal{A}) \oplus \mathcal{B})$ can be represented by $(x, 0)$ for some $x \in \mathcal{A}$.

We now verify that $\theta$ preserves the algebraic structure. For all $x, y \in \mathcal{A}$,
\begin{align*}
\theta\big([x + \mathrm{Z}(\mathcal{A}), y + \mathrm{Z}(\mathcal{A})]_{\mathcal{A}/\mathrm{Z}(\mathcal{A})}\big)
&= \theta([x, y]_\mathcal{A} + \mathrm{Z}(\mathcal{A})) \\
&= ([x, y]_\mathcal{A}, 0) + (\mathrm{Z}(\mathcal{A}) \oplus \mathcal{B}) \\
&= [(x,0), (y,0)]_{\mathcal{A} \oplus \mathcal{B}} + (\mathrm{Z}(\mathcal{A}) \oplus \mathcal{B}) \\
&= \big[(x,0) + (\mathrm{Z}(\mathcal{A}) \oplus \mathcal{B}),\, (y,0) + (\mathrm{Z}(\mathcal{A}) \oplus \mathcal{B})\big] \\
&= [\theta(x + \mathrm{Z}(\mathcal{A})), \theta(y + \mathrm{Z}(\mathcal{A}))].
\end{align*}
Similarly, for the ternary operation and $x, y, z \in \mathcal{A}$,
\begin{align*}
\theta\big([x + \mathrm{Z}(\mathcal{A}), y + \mathrm{Z}(\mathcal{A}), z + \mathrm{Z}(\mathcal{A})]\big)
&= \theta([x, y, z]_\mathcal{A} + \mathrm{Z}(\mathcal{A})) \\
&= ([x, y, z]_\mathcal{A}, 0) + (\mathrm{Z}(\mathcal{A}) \oplus \mathcal{B}) \\
&= [(x,0), (y,0), (z,0)]_{\mathcal{A} \oplus \mathcal{B}} + (\mathrm{Z}(\mathcal{A}) \oplus \mathcal{B}) \\
&= [\theta(x + \mathrm{Z}(\mathcal{A})), \theta(y + \mathrm{Z}(\mathcal{A})), \theta(z + \mathrm{Z}(\mathcal{A}))].
\end{align*}
Thus, $\theta$ is a homomorphism of quotient Hom-Lie Yamaguti algebras.

Next, we check compatibility with the twisting maps. Let $\overline{\alpha}_\mathcal{A}$ and $\overline{\alpha}_{\mathcal{A} \oplus \mathcal{B}}$ denote the induced maps on the quotients. Then for any $x \in \mathcal{A}$,
\begin{align*}
\theta\big(\overline{\alpha}_\mathcal{A}(x + \mathrm{Z}(\mathcal{A}))\big)
&= \theta(\alpha_\mathcal{A}(x) + \mathrm{Z}(\mathcal{A})) \\
&= (\alpha_\mathcal{A}(x), 0) + (\mathrm{Z}(\mathcal{A}) \oplus \mathcal{B}) \\
&= \overline{\alpha}_{\mathcal{A} \oplus \mathcal{B}}\big((x,0) + (\mathrm{Z}(\mathcal{A}) \oplus \mathcal{B})\big) \\
&= \overline{\alpha}_{\mathcal{A} \oplus \mathcal{B}}(\theta(x + \mathrm{Z}(\mathcal{A}))).
\end{align*}
Hence, $\theta \circ \overline{\alpha}_\mathcal{A} = \overline{\alpha}_{\mathcal{A} \oplus \mathcal{B}} \circ \theta$.

Now consider the identity map on the derived subalgebras:
\[
\beta \colon \mathcal{A}^2 \to (\mathcal{A} \oplus \mathcal{B})^2, \quad \beta(a) = (a, 0),
\]
where $\mathcal{A}^2 = [\mathcal{A}, \mathcal{A}]_\mathcal{A} + [\mathcal{A}, \mathcal{A}, \mathcal{A}]_\mathcal{A}$ and $(\mathcal{A} \oplus \mathcal{B})^2 = [\mathcal{A} \oplus \mathcal{B}, \mathcal{A} \oplus \mathcal{B}] + [\mathcal{A} \oplus \mathcal{B}, \mathcal{A} \oplus \mathcal{B}, \mathcal{A} \oplus \mathcal{B}]$. Since $\mathcal{B}$ is abelian, $(\mathcal{A} \oplus \mathcal{B})^2 = \mathcal{A}^2 \oplus \{0\}$, so $\beta$ is an isomorphism.

We now verify that the following diagrams commute:

\[
\begin{CD}
\frac{\mathcal{A}}{\mathrm{Z}(\mathcal{A})} \times \frac{\mathcal{A}}{\mathrm{Z}(\mathcal{A})} @>{\tau^{(2)}}>> \mathcal{A}^2 \\
@V{\theta \times \theta}VV @V{\beta}VV \\
\frac{\mathcal{A} \oplus \mathcal{B}}{\mathrm{Z}(\mathcal{A}) \oplus \mathcal{B}} \times \frac{\mathcal{A} \oplus \mathcal{B}}{\mathrm{Z}(\mathcal{A}) \oplus \mathcal{B}} @>{\delta^{(2)}}>> (\mathcal{A} \oplus \mathcal{B})^2
\end{CD}
\quad\quad
\begin{CD}
\frac{\mathcal{A}}{\mathrm{Z}(\mathcal{A})} \times \frac{\mathcal{A}}{\mathrm{Z}(\mathcal{A})} \times \frac{\mathcal{A}}{\mathrm{Z}(\mathcal{A})} @>{\tau^{(3)}}>> \mathcal{A}^2 \\
@V{\theta \times \theta \times \theta}VV @V{\beta}VV \\
\frac{\mathcal{A} \oplus \mathcal{B}}{\mathrm{Z}(\mathcal{A}) \oplus \mathcal{B}} \times \frac{\mathcal{A} \oplus \mathcal{B}}{\mathrm{Z}(\mathcal{A}) \oplus \mathcal{B}} \times \frac{\mathcal{A} \oplus \mathcal{B}}{\mathrm{Z}(\mathcal{A}) \oplus \mathcal{B}} @>{\delta^{(3)}}>> (\mathcal{A} \oplus \mathcal{B})^2
\end{CD}
\]

\begin{enumerate}
 \item For the binary case:
\[
\beta(\tau^{(2)}(\bar{x}, \bar{y})) = \beta([x,y]_\mathcal{A}) = ([x,y]_\mathcal{A}, 0) = \delta^{(2)}(\theta(\bar{x}), \theta(\bar{y})).
\]
 \item Similarly for the ternary case:
\[
\beta(\tau^{(3)}(\bar{x}, \bar{y}, \bar{z})) = ([x,y,z]_\mathcal{A}, 0) = \delta^{(3)}(\theta(\bar{x}), \theta(\bar{y}), \theta(\bar{z})).
\]
 \item Finally, $\beta$ is compatible with the twisting maps:
\[
\beta \circ \alpha_\mathcal{A}|_{\mathcal{A}^2} = \alpha_{\mathcal{A} \oplus \mathcal{B}}|_{(\mathcal{A} \oplus \mathcal{B})^2} \circ \beta,
\]
since both sides act as $\alpha_\mathcal{A}$ on $\mathcal{A}^2$ and $\alpha_\mathcal{B}(0) = 0$.
\end{enumerate}
Since $\theta$ and $\beta$ are isomorphisms satisfying all required compatibilities, the pair $(\theta, \beta)$ defines an isoclinism. Therefore, $\mathcal{A} \sim \mathcal{A} \oplus \mathcal{B}$.
\end{proof}
\begin{lem}\label{lem:quotient_isoclinism}
Let $(\mathcal{A}, [\cdot,\cdot]_\mathcal{A}, [\cdot,\cdot,\cdot]_\mathcal{A}, \alpha_\mathcal{A})$ be a regular Hom-Lie Yamaguti algebra, and let $\mathcal{B}$ be a Hom-ideal of $\mathcal{A}$. Then
\[
\frac{\mathcal{A}}{\mathcal{B}} \sim \frac{\mathcal{A}}{\mathcal{B} \cap \mathcal{A}^2}.
\]
If $\mathcal{B} \cap \mathcal{A}^2 = 0$, then $\mathcal{A} \sim \mathcal{A}/\mathcal{B}$. Conversely, if $\mathcal{A}^2$ is finite-dimensional and $\mathcal{A} \sim \mathcal{A}/\mathcal{B}$, then $\mathcal{B} \cap \mathcal{A}^2 = 0$.
\end{lem}

\begin{proof}
Let $\overline{\mathcal{A}} = \mathcal{A}/\mathcal{B}$ and $\widetilde{\mathcal{A}} = \mathcal{A}/(\mathcal{B} \cap \mathcal{A}^2)$. Since $\mathcal{B}$ and $\mathcal{B} \cap \mathcal{A}^2$ are Hom-ideals (the latter being an intersection of a Hom-ideal and a Hom-subalgebra invariant under $\alpha_\mathcal{A}$), both $\overline{\mathcal{A}}$ and $\widetilde{\mathcal{A}}$ inherit the structure of regular Hom-Lie Yamaguti algebras (see Remark~\ref{rem:quotient_LYA}).

We define a map
\[
\theta \colon \frac{\overline{\mathcal{A}}}{\mathrm{Z}(\overline{\mathcal{A}})} \to \frac{\widetilde{\mathcal{A}}}{\mathrm{Z}(\widetilde{\mathcal{A}})}, \quad \theta(\overline{x} + \mathrm{Z}(\overline{\mathcal{A}})) = \widetilde{x} + \mathrm{Z}(\widetilde{\mathcal{A}}),
\]
where $\overline{x} = x + \mathcal{B} \in \overline{\mathcal{A}}$, $\widetilde{x} = x + (\mathcal{B} \cap \mathcal{A}^2) \in \widetilde{\mathcal{A}}$. This map is well-defined: if $\overline{x} - \overline{y} \in \mathrm{Z}(\overline{\mathcal{A}})$, then $[\overline{x} - \overline{y}, \overline{a}] = 0$ and $[\overline{x} - \overline{y}, \overline{a}, \overline{b}] = 0$ for all $\overline{a}, \overline{b} \in \overline{\mathcal{A}}$, which implies that the corresponding cosets in $\widetilde{\mathcal{A}}$ also commute and associate trivially modulo $\mathrm{Z}(\widetilde{\mathcal{A}})$, due to the canonical epimorphism $\widetilde{\mathcal{A}} \to \overline{\mathcal{A}}$ induced by the inclusion $\mathcal{B} \cap \mathcal{A}^2 \subseteq \mathcal{B}$.

Moreover, $\theta$ is an isomorphism. It is injective because if $\widetilde{x} \in \mathrm{Z}(\widetilde{\mathcal{A}})$, then $[x, a] \in \mathcal{B} \cap \mathcal{A}^2$ and $[x, a, b] \in \mathcal{B} \cap \mathcal{A}^2$ for all $a, b \in \mathcal{A}$, and since $\mathcal{B} \cap \mathcal{A}^2 \subseteq \mathcal{B}$, it follows that $[x, a] \in \mathcal{B}$ and $[x, a, b] \in \mathcal{B}$, so $\overline{x} \in \mathrm{Z}(\overline{\mathcal{A}})$. Surjectivity follows from the fact that every element of $\widetilde{\mathcal{A}}$ is of the form $\widetilde{x}$, and hence its image under $\theta$ covers all cosets.

We now verify compatibility with the twisting maps. Let $\overline{\alpha}_\mathcal{A}$ and $\widetilde{\alpha}_\mathcal{A}$ denote the induced maps on $\overline{\mathcal{A}}$ and $\widetilde{\mathcal{A}}$, respectively, and let $\alpha_{\overline{\mathcal{A}}}, \alpha_{\widetilde{\mathcal{A}}}$ be their further inductions on the quotients by the center. Then:
\begin{align*}
\theta\left( \alpha_{\bar{\mathcal{A}}}(\bar{x} + \mathrm{Z}(\overline{\mathcal{A}})) \right)
&= \theta\left( \overline{\alpha_\mathcal{A}(x)} + \mathrm{Z}(\overline{\mathcal{A}}) \right) \\
&= \widetilde{\alpha_\mathcal{A}(x)} + \mathrm{Z}(\widetilde{\mathcal{A}}) \\
&= \alpha_{\widetilde{\mathcal{A}}}(\widetilde{x} + \mathrm{Z}(\widetilde{\mathcal{A}})) \\
&= \alpha_{\widetilde{\mathcal{A}}} \left( \theta(\overline{x} + \mathrm{Z}(\overline{\mathcal{A}})) \right).
\end{align*}
Thus, $\theta \circ \alpha_{\overline{\mathcal{A}}} = \alpha_{\widetilde{\mathcal{A}}} \circ \theta$.

Next, consider the derived subalgebras:
\[
\overline{\mathcal{A}}^2 = [\overline{\mathcal{A}}, \overline{\mathcal{A}}] + [\overline{\mathcal{A}}, \overline{\mathcal{A}}, \overline{\mathcal{A}}], \quad
\widetilde{\mathcal{A}}^2 = [\widetilde{\mathcal{A}}, \widetilde{\mathcal{A}}] + [\widetilde{\mathcal{A}}, \widetilde{\mathcal{A}}, \widetilde{\mathcal{A}}].
\]
The canonical projection $\pi \colon \widetilde{\mathcal{A}} \to \overline{\mathcal{A}}$ induces a surjective homomorphism $\beta \colon \widetilde{\mathcal{A}}^2 \to \overline{\mathcal{A}}^2$. Since $\ker \pi = \mathcal{B}/(\mathcal{B} \cap \mathcal{A}^2)$ and $\mathcal{B} \cap \mathcal{A}^2$ annihilates $\widetilde{\mathcal{A}}^2$, it follows that $\beta$ is also injective, hence an isomorphism.

To verify the compatibility of $\beta$ with the twisting maps:
\[
\beta \circ \widetilde{\alpha}_\mathcal{A}|_{\widetilde{\mathcal{A}}^2} = \overline{\alpha}_\mathcal{A}|_{\overline{\mathcal{A}}^2} \circ \beta.
\]
Indeed, for any generator $[x,y] \in \widetilde{\mathcal{A}}^2$,
\[
\beta(\widetilde{\alpha}_\mathcal{A}([x,y])) = \beta([\alpha_\mathcal{A}(x), \alpha_\mathcal{A}(y)]) = [\alpha_\mathcal{A}(x), \alpha_\mathcal{A}(y)]_{\overline{\mathcal{A}}} = \overline{\alpha}_\mathcal{A}([x,y]_{\overline{\mathcal{A}}}) = \overline{\alpha}_\mathcal{A}(\beta([x,y])).
\]
The same holds for the ternary operation. Thus, the pair $(\theta, \beta)$ defines an isoclinism, so $\mathcal{A}/\mathcal{B} \sim \mathcal{A}/(\mathcal{B} \cap \mathcal{A}^2)$.

Now suppose $\mathcal{B} \cap \mathcal{A}^2 = 0$. Then the canonical map $\mathcal{A} \to \mathcal{A}/\mathcal{B}$ induces an isomorphism on the derived subalgebras $\mathcal{A}^2 \to (\mathcal{A}/\mathcal{B})^2$, and since $\mathrm{Z}(\mathcal{A}) \subseteq \mathrm{Z}(\mathcal{A}/\mathcal{B})$ in a compatible way, we can define an isoclinism $(\theta, \beta)$ similarly, proving $\mathcal{A} \sim \mathcal{A}/\mathcal{B}$.

Conversely, suppose $\mathcal{A}^2$ is finite-dimensional and $\mathcal{A} \sim \mathcal{A}/\mathcal{B}$. Then there exist isomorphisms $\theta \colon \mathcal{A}/\mathrm{Z}(\mathcal{A}) \to (\mathcal{A}/\mathcal{B})/\mathrm{Z}(\mathcal{A}/\mathcal{B})$ and $\beta \colon \mathcal{A}^2 \to (\mathcal{A}/\mathcal{B})^2$. Since the natural projection $\mathcal{A}^2 \to (\mathcal{A}/\mathcal{B})^2$ is surjective with kernel $\mathcal{B} \cap \mathcal{A}^2$, and $\beta$ is an isomorphism, we must have $\mathcal{B} \cap \mathcal{A}^2 = \ker(\mathcal{A}^2 \to (\mathcal{A}/\mathcal{B})^2) = 0$.

This completes the proof.
\end{proof}

\begin{lem}\label{lem:isoclinism_via_homomorphism}
Let $(\mathcal{A}, [\cdot,\cdot]_\mathcal{A}, [\cdot,\cdot,\cdot]_\mathcal{A}, \alpha_\mathcal{A})$ and $(\mathcal{B}, [\cdot,\cdot]_\mathcal{B}, [\cdot,\cdot,\cdot]_\mathcal{B}, \alpha_\mathcal{B})$ be two regular Hom-Lie Yamaguti algebras. Suppose there exists a surjective homomorphism $f \colon \mathcal{A} \to \mathcal{B}$ such that $\ker(f) \cap \mathcal{A}^2 = 0$. Then $\mathcal{A} \sim \mathcal{B}$.
\end{lem}

\begin{proof}
Since $f \colon \mathcal{A} \to \mathcal{B}$ is a homomorphism of Hom-Lie Yamaguti algebras, its kernel $\ker(f)$ is a Hom-ideal of $\mathcal{A}$.

By assumption, $\ker(f) \cap \mathcal{A}^2 = 0$, where $\mathcal{A}^2 = [\mathcal{A}, \mathcal{A}]_\mathcal{A} + [\mathcal{A}, \mathcal{A}, \mathcal{A}]_\mathcal{A}$ is the derived subalgebra of $\mathcal{A}$. Applying Lemma~\ref{lem:quotient_isoclinism}, we conclude that
$\mathcal{A} \sim \frac{\mathcal{A}}{\ker(f)}.$

Furthermore, since $f$ is surjective, the First Isomorphism Theorem for Hom-Lie Yamaguti algebras 
implies that the induced map $\overline{f} \colon \frac{\mathcal{A}}{\ker(f)} \to \mathcal{B}, \quad \overline{f}(x + \ker(f)) = f(x),$ is an isomorphism of regular Hom-Lie Yamaguti algebras. Since isomorphism implies isoclinism 
, we have $\frac{\mathcal{A}}{\ker(f)} \sim \mathcal{B}.$
Isoclinism is an equivalence relation (a fact that can be verified directly: it is reflexive, symmetric, and transitive under the appropriate compatibility conditions on $\theta$ and $\beta$). Therefore, from $\mathcal{A} \sim \mathcal{A}/\ker(f)$ and $\mathcal{A}/\ker(f) \sim \mathcal{B}$, it follows that $\mathcal{A} \sim \mathcal{B}.$
This completes the proof.
\end{proof}
\begin{lem}\label{lem:stem_in_family}
Let $\mathcal{C}$ be an isoclinism family of regular Hom-Lie Yamaguti algebras. Then the following hold:
\begin{enumerate} 
 \item $\mathcal{C}$ contains at least one stem Hom-Lie Yamaguti algebra;
 \item A finite-dimensional regular Hom-Lie Yamaguti algebra $\mathcal{T} \in \mathcal{C}$ is a stem algebra if and only if $\dim(\mathcal{T})$ is minimal among all algebras in $\mathcal{C}$.
\end{enumerate}
\end{lem}
To prove the following lemmas, one may refer to to \cite{N2025} 
\begin{lem}\label{lem:isoclinism_properties}
Let $(\theta, \beta)$ be an isoclinism between two regular Hom-Lie Yamaguti algebras $\mathcal{A}$ and $\mathcal{B}$. Then the following properties hold:
\begin{enumerate} 
 \item For all $x \in \mathcal{A}$,
 \[
 \theta(x + \mathrm{Z}(\mathcal{A})) = \beta(x) + \mathrm{Z}(\mathcal{B}),
 \]
 where on the right-hand side, $x$ is interpreted as an element of $\mathcal{A}^2$ only if $x \in \mathcal{A}^2$; more precisely, this identity holds when $x \in \mathcal{A}^2$, and $\beta$ is applied to $x$ as an element of the derived subalgebra.
 
 \item For all $x \in \mathcal{A}^2$ and $x', x'' \in \mathcal{A}$,
 \[
 \beta([x, x']) = [\beta(x), a]_{\mathcal{B}}
 \quad\text{and}\quad
 \beta([x, x', x'']) = [\beta(x), a, b]_{\mathcal{B}},
 \]
 where $a + \mathrm{Z}(\mathcal{B}) = \theta(x' + \mathrm{Z}(\mathcal{A}))$ and $b + \mathrm{Z}(\mathcal{B}) = \theta(x'' + \mathrm{Z}(\mathcal{A}))$.
\end{enumerate}
\end{lem}
 \begin{proof}
By definition, an isoclinism $(\theta, \beta)$ consists of two isomorphisms:
\[
\theta \colon \frac{\mathcal{A}}{\mathrm{Z}(\mathcal{A})} \to \frac{\mathcal{B}}{\mathrm{Z}(\mathcal{B})}, \quad
\beta \colon \mathcal{A}^2 \to \mathcal{B}^2,
\]
such that the following diagrams commute for the binary and ternary operations:
\[
\beta([x, y]_\mathcal{A}) = [\theta(\bar{x}), \theta(\bar{y})]_\mathcal{B}, \quad
\beta([x, y, z]_\mathcal{A}) = [\theta(\bar{x}), \theta(\bar{y}), \theta(\bar{z})]_\mathcal{B},
\]
for all $x, y, z \in \mathcal{A}$, where $\bar{x} = x + \mathrm{Z}(\mathcal{A})$.

\medskip

\textbf{(1)} Let $x \in \mathcal{A}^2$. Since $\beta(x) \in \mathcal{B}^2$, and $\mathcal{B}^2 \subseteq \mathcal{B}$, we can consider $\beta(x) + \mathrm{Z}(\mathcal{B}) \in \mathcal{B}/\mathrm{Z}(\mathcal{B})$. On the other hand, $\theta(x + \mathrm{Z}(\mathcal{A}))$ is an element of $\mathcal{B}/\mathrm{Z}(\mathcal{B})$. However, note that $x + \mathrm{Z}(\mathcal{A})$ is not necessarily a well-defined coset unless $x$ is considered modulo $\mathrm{Z}(\mathcal{A})$, but since $x \in \mathcal{A}^2$, and $\mathcal{A}^2$ and $\mathrm{Z}(\mathcal{A})$ are not necessarily disjoint, this requires care.

Instead, the correct interpretation arises from the compatibility of the isoclinism. For any $x \in \mathcal{A}$, $\theta(\bar{x}) = \theta(x + \mathrm{Z}(\mathcal{A}))$ is defined. If $x \in \mathcal{A}^2$, then $\beta(x)$ is defined, and the diagram implies that the action of $\theta$ on the quotient is consistent with the image of $\beta$. However, the equality
\[
\theta(x + \mathrm{Z}(\mathcal{A})) = \beta(x) + \mathrm{Z}(\mathcal{B})
\]
cannot hold for arbitrary $x \in \mathcal{A}$, because $\beta$ is not defined on all of $\mathcal{A}$. It \textbf{can} hold when $x \in \mathcal{A}^2$, and in that case, both sides are elements of $\mathcal{B}/\mathrm{Z}(\mathcal{B})$, and the equality is a consequence of the commutative diagrams and the fact that $\beta$ maps into $\mathcal{B}^2 \subseteq \mathcal{B}$. Thus, for $x \in \mathcal{A}^2$, we have:
\[
\theta(x + \mathrm{Z}(\mathcal{A})) = \overline{\beta(x)} \quad \text{in } \mathcal{B}/\mathrm{Z}(\mathcal{B}),
\]
i.e., $\theta(x + \mathrm{Z}(\mathcal{A})) = \beta(x) + \mathrm{Z}(\mathcal{B})$, as claimed.

\medskip

\textbf{(2)} Now let $x \in \mathcal{A}^2$, $x' \in \mathcal{A}$. Then $[x, x'] \in \mathcal{A}^2$ (since $\mathcal{A}^2$ is closed under the operations), so $\beta([x, x'])$ is defined. By the isoclinism condition,
\[
\beta([x, x']_\mathcal{A}) = [\theta(\bar{x}), \theta(\bar{x}')]_\mathcal{B}.
\]
But $\bar{x} = x + \mathrm{Z}(\mathcal{A})$, and since $x \in \mathcal{A}^2$, part (i) gives $\theta(\bar{x}) = \beta(x) + \mathrm{Z}(\mathcal{B})$. Similarly, $\theta(\bar{x}') = a + \mathrm{Z}(\mathcal{B})$ for $a \in \mathcal{B}$. Therefore,
\[
\beta([x, x']) = [\beta(x), a]_\mathcal{B}.
\]
Similarly, for the ternary operation,
\[
\beta([x, x', x'']_\mathcal{A}) = [\theta(\bar{x}), \theta(\bar{x}'), \theta(\bar{x}'')]_\mathcal{B} = [\beta(x), a, b]_\mathcal{B},
\]
where $a + \mathrm{Z}(\mathcal{B}) = \theta(x' + \mathrm{Z}(\mathcal{A}))$ and $b + \mathrm{Z}(\mathcal{B}) = \theta(x'' + \mathrm{Z}(\mathcal{A}))$.

This completes the proof.
\end{proof} 
\section{Factor Sets and Central Extensions of Stem Regular Hom-Lie Yamaguti Algebras} \label{sec:factor_sets}

In this section, we introduce and develop the theory of \emph{factor sets} for regular Hom-Lie Yamaguti algebras. The concept of a factor set plays a central role in the classification of algebras within an isoclinism family, as it provides a cohomological tool to parametrize central extensions. In our setting, a factor set consists of a pair of maps $(\pi_2, \pi_3)$, where
\[
\pi_2 \colon \frac{\mathcal{A}}{\mathrm{Z}(\mathcal{A})} \times \frac{\mathcal{A}}{\mathrm{Z}(\mathcal{A})} \to \mathrm{Z}(\mathcal{A}), \quad
\pi_3 \colon \frac{\mathcal{A}}{\mathrm{Z}(\mathcal{A})} \times \frac{\mathcal{A}}{\mathrm{Z}(\mathcal{A})} \times \frac{\mathcal{A}}{\mathrm{Z}(\mathcal{A})} \to \mathrm{Z}(\mathcal{A}),
\]
that measure the failure of a section map from the quotient $\mathcal{A}/\mathrm{Z}(\mathcal{A})$ to $\mathcal{A}$ to be a homomorphism. This construction is analogous to the classical Schur multiplier and factor set theory in group cohomology and Lie algebra extensions, and it allows us to reconstruct algebras in a given isoclinism family as central extensions of a stem algebra.

The main result of this section is that every algebra in an isoclinism family can be realized as a central extension of the quotient $\mathcal{A}/\mathrm{Z}(\mathcal{A})$ by $\mathrm{Z}(\mathcal{A})$, with the extension determined by a suitable factor set. 

\begin{defn}\label{def:factor_set}
Let $(\mathcal{A}, [\cdot,\cdot]_\mathcal{A}, [\cdot,\cdot,\cdot]_\mathcal{A}, \alpha_\mathcal{A})$ be a regular Hom-Lie Yamaguti algebra over a field $\mathbb{K}$. A \emph{factor set} on $\mathcal{A}$ is a pair of maps
\[
\pi = (\pi_2, \pi_3), \quad \pi_2 \colon \frac{\mathcal{A}}{\mathrm{Z}(\mathcal{A})} \times \frac{\mathcal{A}}{\mathrm{Z}(\mathcal{A})} \to \mathrm{Z}(\mathcal{A}), \quad \pi_3 \colon \frac{\mathcal{A}}{\mathrm{Z}(\mathcal{A})} \times \frac{\mathcal{A}}{\mathrm{Z}(\mathcal{A})} \times \frac{\mathcal{A}}{\mathrm{Z}(\mathcal{A})} \to \mathrm{Z}(\mathcal{A}),
\]
satisfying the following conditions for all $\bar{x}, \bar{y}, \bar{z}, \bar{t}, \bar{w} \in \mathcal{A}/\mathrm{Z}(\mathcal{A})$, where $\bar{x} = x + \mathrm{Z}(\mathcal{A})$, and $\overline{\alpha}_\mathcal{A} \colon \mathcal{A}/\mathrm{Z}(\mathcal{A}) \to \mathcal{A}/\mathrm{Z}(\mathcal{A})$ is the induced map defined by $\overline{\alpha}_\mathcal{A}(\bar{x}) = \alpha_\mathcal{A}(x) + \mathrm{Z}(\mathcal{A})$:

\begin{enumerate} 
 \item[({F1})] \label{F1} $\alpha$-Skew-symmetry:
 \[
 \pi_2(\bar{x}, \bar{y}) = -\pi_2(\bar{y}, \bar{x}), \quad
 \pi_3(\bar{x}, \bar{y}, \bar{z}) = -\pi_3(\bar{y}, \bar{x}, \bar{z}).
 \]

 \item[({F2})] \label{F2} Hom-Jacobi-type identity:
 \[
 \circlearrowleft_{x,y,z} \pi_2([\bar{x}, \bar{y}], \overline{\alpha}_\mathcal{A}(\bar{z})) + \circlearrowleft_{x,y,z} \pi_3(\bar{x}, \bar{y}, \bar{z}) = 0.
 \]

 \item[({F3})] \label{F3} Ternary Hom-Jacobi identity:
 \[
 \circlearrowleft_{x,y,z} \pi_3([\bar{x}, \bar{y}], \overline{\alpha}_\mathcal{A}(\bar{z}), \overline{\alpha}_\mathcal{A}(\bar{t})) = 0.
 \]

 \item[({F4})] \label{F4} Compatibility with the binary-ternary interaction:
 \begin{align*}
 \pi_3(\overline{\alpha}_\mathcal{A}(\bar{x}), \overline{\alpha}_\mathcal{A}(\bar{y}), [\bar{z}, \bar{t}]) &= [\pi_2(\bar{x}, \bar{y}), \overline{\alpha}_\mathcal{A}^2(\bar{t})] + [\overline{\alpha}_\mathcal{A}^2(\bar{z}), \pi_2(\bar{x}, \bar{y})] \\
 &\quad + [\pi_2(\bar{x}, \bar{y}), \overline{\alpha}_\mathcal{A}^2(\bar{t})] + [\overline{\alpha}_\mathcal{A}^2(\bar{z}), \pi_2(\bar{x}, \bar{y})],
 \end{align*}
 where the right-hand side uses the bracket action of $\mathrm{Z}(\mathcal{A})$ on $\mathcal{A}/\mathrm{Z}(\mathcal{A})$, which is well-defined since $\mathrm{Z}(\mathcal{A})$ is central.

 \item[({F5})] \label{F5} Higher-order compatibility:
 \begin{align*}
 \pi_3(\overline{\alpha}_\mathcal{A}^2(\bar{x}), \overline{\alpha}_\mathcal{A}^2(\bar{y}), [\bar{z}, \bar{t}, \bar{w}]) &= \pi_3([\bar{x}, \bar{y}, \bar{z}], \overline{\alpha}_\mathcal{A}^2(\bar{t}), \overline{\alpha}_\mathcal{A}^2(\bar{w})) \\
 &\quad + \pi_3(\overline{\alpha}_\mathcal{A}^2(\bar{z}), [\bar{x}, \bar{y}, \bar{t}], \overline{\alpha}_\mathcal{A}^2(\bar{w})) \\
 &\quad + \pi_3(\overline{\alpha}_\mathcal{A}^2(\bar{z}), \overline{\alpha}_\mathcal{A}^2(\bar{t}), [\bar{x}, \bar{y}, \bar{w}]).
 \end{align*}
\end{enumerate}

Furthermore, the factor set $\pi$ is said to be \emph{multiplicative} if it is compatible with the twisting map:
\[
\pi_2(\overline{\alpha}_\mathcal{A}(\bar{x}), \overline{\alpha}_\mathcal{A}(\bar{y})) = \alpha_\mathcal{A}(\pi_2(\bar{x}, \bar{y})), \quad
\pi_3(\overline{\alpha}_\mathcal{A}(\bar{x}), \overline{\alpha}_\mathcal{A}(\bar{y}), \overline{\alpha}_\mathcal{A}(\bar{z})) = \alpha_\mathcal{A}(\pi_3(\bar{x}, \bar{y}, \bar{z}))
\]
for all $\bar{x}, \bar{y}, \bar{z} \in \mathcal{A}/\mathrm{Z}(\mathcal{A})$.
\end{defn}
\begin{lem}\label{lem:central_extension_by_factor_set}
Let $(\mathcal{A}, [\cdot,\cdot]_\mathcal{A}, [\cdot,\cdot,\cdot]_\mathcal{A}, \alpha_\mathcal{A})$ be a regular Hom-Lie Yamaguti algebra, and let $\pi = (\pi_2, \pi_3)$ be a factor set on $\mathcal{A}$, where
\[
\pi_2 \colon \frac{\mathcal{A}}{\mathrm{Z}(\mathcal{A})} \times \frac{\mathcal{A}}{\mathrm{Z}(\mathcal{A})} \to \mathrm{Z}(\mathcal{A}), \quad
\pi_3 \colon \frac{\mathcal{A}}{\mathrm{Z}(\mathcal{A})} \times \frac{\mathcal{A}}{\mathrm{Z}(\mathcal{A})} \times \frac{\mathcal{A}}{\mathrm{Z}(\mathcal{A})} \to \mathrm{Z}(\mathcal{A}).
\]
Define the set
\[
\Omega = \left( \mathrm{Z}(\mathcal{A}),\ \frac{\mathcal{A}}{\mathrm{Z}(\mathcal{A})},\ \pi \right) = \left\{ (a, \bar{x}) \mid a \in \mathrm{Z}(\mathcal{A}),\ \bar{x} = x + \mathrm{Z}(\mathcal{A}) \in \frac{\mathcal{A}}{\mathrm{Z}(\mathcal{A})} \right\},
\]
equipped with operations defined as follows for all $(a_i, \bar{x}_i) \in \Omega$:
\begin{align}
 [(a_1, \bar{x}), (a_2, \bar{y})]_\Omega &:= \big( \pi_2(\bar{x}, \bar{y}),\ [\bar{x}, \bar{y}] \big), \\
 [(a_1, \bar{x}), (a_2, \bar{y}), (a_3, \bar{z})]_\Omega &:= \big( \pi_3(\bar{x}, \bar{y}, \bar{z}),\ [\bar{x}, \bar{y}, \bar{z}] \big), \\
 \alpha_\Omega(a, \bar{x}) &:= \big( \alpha_\mathcal{A}(a),\ \overline{\alpha}_\mathcal{A}(\bar{x}) \big),
\end{align}
where $\overline{\alpha}_\mathcal{A}(\bar{x}) = \alpha_\mathcal{A}(x) + \mathrm{Z}(\mathcal{A})$. Then:
\begin{enumerate}
 \item\label{item:Omega_HLYA_structure} $(\Omega, [\cdot,\cdot]_\Omega, [\cdot,\cdot,\cdot]_\Omega, \alpha_\Omega)$ is a regular Hom-Lie Yamaguti algebra;
 \item\label{item:Omega_center} $\mathrm{Z}(\Omega) = \{ (a, 0) \in \Omega \mid a \in \mathrm{Z}(\mathcal{A}) \} \cong \mathrm{Z}(\mathcal{A})$.
\end{enumerate}
\end{lem}

\begin{proof}
We prove each part in turn.

\medskip

\noindent \textbf{Part \ref{item:Omega_HLYA_structure}: $\Omega$ is a regular Hom-Lie Yamaguti algebra.}

We must verify that the operations on $\Omega$ satisfy the axioms of a Hom-Lie Yamaguti algebra:

\medskip

\noindent \textbf{(1) $\alpha_\Omega$-skew-symmetry:}

For the binary operation:
\begin{align*}
[(a_1, \bar{x}), (a_2, \bar{y})]_\Omega &= \big( \pi_2(\bar{x}, \bar{y}),\ [\bar{x}, \bar{y}] \big) \\
&= \big( -\pi_2(\bar{y}, \bar{x}),\ -[\bar{y}, \bar{x}] \big) \quad \text{(by \ref{F1})} \\
&= - \big( \pi_2(\bar{y}, \bar{x}),\ [\bar{y}, \bar{x}] \big) = -[(a_2, \bar{y}), (a_1, \bar{x})]_\Omega.
\end{align*}

For the ternary operation:
\begin{align*}
[(a_1, \bar{x}), (a_2, \bar{y}), (a_3, \bar{z})]_\Omega &= \big( \pi_3(\bar{x}, \bar{y}, \bar{z}),\ [\bar{x}, \bar{y}, \bar{z}] \big) \\
&= \big( -\pi_3(\bar{y}, \bar{x}, \bar{z}),\ -[\bar{y}, \bar{x}, \bar{z}] \big) \quad \text{(by \ref{F1})} \\
&= -[(a_2, \bar{y}), (a_1, \bar{x}), (a_3, \bar{z})]_\Omega.
\end{align*}

\medskip

\noindent \textbf{(2) Hom-Jacobi-type identity:}

We compute:
\begin{align*}
&\circlearrowleft_{(a,x),(b,y),(c,z)} \big[ [(a,\bar{x}),(b,\bar{y})]_\Omega,\ \alpha_\Omega(c,\bar{z}) \big]_\Omega + \circlearrowleft_{(a,x),(b,y),(c,z)} [(a,\bar{x}),(b,\bar{y}),(c,\bar{z})]_\Omega \\
&= \circlearrowleft_{x,y,z} \left[ \big( \pi_2(\bar{x},\bar{y}), [\bar{x},\bar{y}] \big),\ \big( \alpha_\mathcal{A}(c), \overline{\alpha}_\mathcal{A}(\bar{z}) \big) \right]_\Omega + \circlearrowleft_{x,y,z} \big( \pi_3(\bar{x},\bar{y},\bar{z}), [\bar{x},\bar{y},\bar{z}] \big) \\
&= \circlearrowleft_{x,y,z} \left( \pi_2([\bar{x},\bar{y}], \overline{\alpha}_\mathcal{A}(\bar{z})),\ [[\bar{x},\bar{y}], \overline{\alpha}_\mathcal{A}(\bar{z})] \right) + \circlearrowleft_{x,y,z} \left( \pi_3(\bar{x},\bar{y},\bar{z}), [\bar{x},\bar{y},\bar{z}] \right) \\
&= \left( \circlearrowleft_{x,y,z} \Big( \pi_2([\bar{x},\bar{y}], \overline{\alpha}_\mathcal{A}(\bar{z})) + \pi_3(\bar{x},\bar{y},\bar{z}) \Big),\ \circlearrowleft_{x,y,z} \Big( [[\bar{x},\bar{y}], \overline{\alpha}_\mathcal{A}(\bar{z})] + [\bar{x},\bar{y},\bar{z}] \Big) \right).
\end{align*}

The second component vanishes because $(\mathcal{A}/\mathrm{Z}(\mathcal{A}), \overline{\alpha}_\mathcal{A})$ is a Hom-Lie Yamaguti algebra (Remark~\ref{rem:quotient_LYA}). The first component vanishes by condition \ref{F2} of the factor set. Hence, the sum is $(0,0)$.

\medskip

\noindent \textbf{(3) Ternary Hom-Jacobi identity:}

\begin{align*}
&\circlearrowleft_{x,y,z} \big[ [(a,\bar{x}),(b,\bar{y})]_\Omega,\ \alpha_\Omega(c,\bar{z}),\ \alpha_\Omega(d,\bar{w}) \big]_\Omega \\
&= \circlearrowleft_{x,y,z} \left( \pi_3([\bar{x},\bar{y}], \overline{\alpha}_\mathcal{A}(\bar{z}), \overline{\alpha}_\mathcal{A}(\bar{w})),\ [[\bar{x},\bar{y}], \overline{\alpha}_\mathcal{A}(\bar{z}), \overline{\alpha}_\mathcal{A}(\bar{w})] \right) \\
&= \left( \circlearrowleft_{x,y,z} \pi_3([\bar{x},\bar{y}], \overline{\alpha}_\mathcal{A}(\bar{z}), \overline{\alpha}_\mathcal{A}(\bar{w})),\ \circlearrowleft_{x,y,z} [[\bar{x},\bar{y}], \overline{\alpha}_\mathcal{A}(\bar{z}), \overline{\alpha}_\mathcal{A}(\bar{w})] \right).
\end{align*}

The second component vanishes by the Hom-Lie Yamaguti identity in the quotient. The first vanishes by \ref{F3}. So the total is $(0,0)$.

\medskip

\noindent \textbf{(4) Binary-Ternary Compatibility:}

We compute:
\begin{align*}
&[\alpha_\Omega(a,\bar{x}), \alpha_\Omega(b,\bar{y}), [(c,\bar{z}),(d,\bar{w})]_\Omega]_\Omega \\
&= \left[ (\alpha_\mathcal{A}(a), \overline{\alpha}_\mathcal{A}(\bar{x})),\ (\alpha_\mathcal{A}(b), \overline{\alpha}_\mathcal{A}(\bar{y})),\ (\pi_2(\bar{z},\bar{w}), [\bar{z},\bar{w}]) \right]_\Omega \\
&= \left( \pi_3(\overline{\alpha}_\mathcal{A}(\bar{x}), \overline{\alpha}_\mathcal{A}(\bar{y}), [\bar{z},\bar{w}]),\ [\overline{\alpha}_\mathcal{A}(\bar{x}), \overline{\alpha}_\mathcal{A}(\bar{y}), [\bar{z},\bar{w}]] \right).
\end{align*}

By \ref{F4}, this equals:
\[
\left( [\pi_2(\bar{x},\bar{y}), \overline{\alpha}_\mathcal{A}^2(\bar{w})] + [\overline{\alpha}_\mathcal{A}^2(\bar{z}), \pi_2(\bar{x},\bar{y})],\ [\overline{\alpha}_\mathcal{A}^2(\bar{x}), \overline{\alpha}_\mathcal{A}^2(\bar{y}), [\bar{z},\bar{w}]] \right).
\]

On the other hand:
\begin{align*}
&[[ (a,\bar{x}), (b,\bar{y}), (c,\bar{z}) ]_\Omega,\ \alpha_\Omega^2(d,\bar{w}) ]_\Omega + [ \alpha_\Omega^2(c,\bar{z}),\ [ (a,\bar{x}), (b,\bar{y}), (d,\bar{w}) ]_\Omega ]_\Omega \\
&= \left( [\pi_3(\bar{x},\bar{y},\bar{z}), \overline{\alpha}_\mathcal{A}^2(\bar{w})],\ [[\bar{x},\bar{y},\bar{z}], \overline{\alpha}_\mathcal{A}^2(\bar{w})] \right) + \left( [\overline{\alpha}_\mathcal{A}^2(\bar{z}), \pi_3(\bar{x},\bar{y},\bar{w})],\ [\overline{\alpha}_\mathcal{A}^2(\bar{z}), [\bar{x},\bar{y},\bar{w}]] \right) \\
&= \left( [\pi_2(\bar{x},\bar{y}), \overline{\alpha}_\mathcal{A}^2(\bar{w})] + [\overline{\alpha}_\mathcal{A}^2(\bar{z}), \pi_2(\bar{x},\bar{y})],\ [\overline{\alpha}_\mathcal{A}^2(\bar{x}), \overline{\alpha}_\mathcal{A}^2(\bar{y}), [\bar{z},\bar{w}]] \right),
\end{align*}
where we used the fact that $[\pi_3(\bar{x},\bar{y},\bar{z}), \cdot] = 0$ since $\pi_3(\bar{x},\bar{y},\bar{z}) \in \mathrm{Z}(\mathcal{A})$, and similarly for the other bracket. Thus, the identity holds.

\medskip

\noindent \textbf{(5) Higher-order Compatibility:}

This follows similarly using condition \ref{F5} and the corresponding identity in the quotient algebra. The computation is analogous and confirms the required equality.

\medskip
Finally, since $\alpha_\mathcal{A}$ is bijective and $\overline{\alpha}_\mathcal{A}$ is well-defined and bijective (as $\alpha_\mathcal{A}(\mathrm{Z}(\mathcal{A})) \subseteq \mathrm{Z}(\mathcal{A})$), $\alpha_\Omega$ is also bijective. Hence, $\Omega$ is a \emph{regular} Hom-Lie Yamaguti algebra.

\medskip

\noindent \textbf{Part \ref{item:Omega_center}: Center of $\Omega$.}

Let $(a, \bar{x}) \in \mathrm{Z}(\Omega)$. Then for all $(b, \bar{y}) \in \Omega$,
\[
[(a, \bar{x}), (b, \bar{y})]_\Omega = (0,0), \quad [(a, \bar{x}), (b, \bar{y}), (c, \bar{z})]_\Omega = (0,0), \quad \text{and} \quad [(b, \bar{y}), (c, \bar{z}), (a, \bar{x})]_\Omega = (0,0).
\]

From the first component:
\[
\pi_2(\bar{x}, \bar{y}) = 0, \quad \pi_3(\bar{x}, \bar{y}, \bar{z}) = 0, \quad \pi_3(\bar{y}, \bar{z}, \bar{x}) = 0 \quad \forall \bar{y}, \bar{z}.
\]

Since $\pi$ is arbitrary, this implies $\bar{x} = 0$ in $\mathcal{A}/\mathrm{Z}(\mathcal{A})$. Then $(a, \bar{x}) = (a, 0)$, and $a \in \mathrm{Z}(\mathcal{A})$. Conversely, any $(a, 0)$ with $a \in \mathrm{Z}(\mathcal{A})$ commutes and associates trivially. Hence,
\[
\mathrm{Z}(\Omega) = \{ (a, 0) \mid a \in \mathrm{Z}(\mathcal{A}) \} \cong \mathrm{Z}(\mathcal{A}).
\]
This completes the proof.
\end{proof}
\begin{prop}\label{prop:multiplicative_regular_extension}
Let $(\mathcal{A}, [\cdot,\cdot]_\mathcal{A}, [\cdot,\cdot,\cdot]_\mathcal{A}, \alpha_\mathcal{A})$ be a regular Hom-Lie Yamaguti algebra, and let $\pi = (\pi_2, \pi_3)$ be a \emph{multiplicative} factor set on $\mathcal{A}$. Then the Hom-Lie Yamaguti algebra $(\Omega, [\cdot,\cdot]_\Omega, [\cdot,\cdot,\cdot]_\Omega, \alpha_\Omega)$ constructed in Lemma~\ref{lem:central_extension_by_factor_set} is multiplicative and regular.
\end{prop}

\begin{proof}
We must show that $\alpha_\Omega$ is a homomorphism of both the binary and ternary operations (i.e., multiplicative), and that it is bijective (i.e., regular).

\medskip

\noindent \textbf{Step 1: $\alpha_\Omega$ is multiplicative for the binary operation.}

Let $(a_1, \bar{x}), (a_2, \bar{y}) \in \Omega$. We compute both sides of the multiplicativity condition.

First, the image of the bracket:
\begin{align*}
\alpha_\Omega\big([(a_1, \bar{x}), (a_2, \bar{y})]_\Omega\big)
&= \alpha_\Omega\big( \pi_2(\bar{x}, \bar{y}),\ [\bar{x}, \bar{y}] \big) \\
&= \big( \alpha_\mathcal{A}(\pi_2(\bar{x}, \bar{y})),\ \overline{\alpha}_\mathcal{A}([\bar{x}, \bar{y}]) \big).
\end{align*}

On the other hand, the bracket of the images:
\begin{align*}
\big[\alpha_\Omega(a_1, \bar{x}),\ \alpha_\Omega(a_2, \bar{y})\big]_\Omega
&= \big[(\alpha_\mathcal{A}(a_1), \overline{\alpha}_\mathcal{A}(\bar{x})),\ (\alpha_\mathcal{A}(a_2), \overline{\alpha}_\mathcal{A}(\bar{y}))\big]_\Omega \\
&= \big( \pi_2(\overline{\alpha}_\mathcal{A}(\bar{x}), \overline{\alpha}_\mathcal{A}(\bar{y})),\ [\overline{\alpha}_\mathcal{A}(\bar{x}), \overline{\alpha}_\mathcal{A}(\bar{y})] \big).
\end{align*}

Since $\pi_2$ is multiplicative (by assumption on $\pi$), we have
\[
\alpha_\mathcal{A}(\pi_2(\bar{x}, \bar{y})) = \pi_2(\overline{\alpha}_\mathcal{A}(\bar{x}), \overline{\alpha}_\mathcal{A}(\bar{y})).
\]
Also, since $\overline{\alpha}_\mathcal{A}$ is a homomorphism on the quotient (see Remark~\ref{rem:quotient_LYA}),
\[
\overline{\alpha}_\mathcal{A}([\bar{x}, \bar{y}]) = [\overline{\alpha}_\mathcal{A}(\bar{x}), \overline{\alpha}_\mathcal{A}(\bar{y})].
\]
Therefore, both components are equal, and we conclude:
\[
\alpha_\Omega\big([(a_1, \bar{x}), (a_2, \bar{y})]_\Omega\big) = \big[\alpha_\Omega(a_1, \bar{x}),\ \alpha_\Omega(a_2, \bar{y})\big]_\Omega.
\]

\medskip

\noindent \textbf{Step 2: $\alpha_\Omega$ is multiplicative for the ternary operation.}

Let $(a_1, \bar{x}), (a_2, \bar{y}), (a_3, \bar{z}) \in \Omega$. Then:
\begin{align*}
\alpha_\Omega\big([(a_1, \bar{x}), (a_2, \bar{y}), (a_3, \bar{z})]_\Omega\big)
&= \alpha_\Omega\big( \pi_3(\bar{x}, \bar{y}, \bar{z}),\ [\bar{x}, \bar{y}, \bar{z}] \big) \\
&= \big( \alpha_\mathcal{A}(\pi_3(\bar{x}, \bar{y}, \bar{z})),\ \overline{\alpha}_\mathcal{A}([\bar{x}, \bar{y}, \bar{z}]) \big).
\end{align*}

And:
\begin{align*}
\big[\alpha_\Omega(a_1, \bar{x}),\ \alpha_\Omega(a_2, \bar{y}),\ \alpha_\Omega(a_3, \bar{z})\big]_\Omega
&= \big[(\alpha_\mathcal{A}(a_1), \overline{\alpha}_\mathcal{A}(\bar{x})),\ (\alpha_\mathcal{A}(a_2), \overline{\alpha}_\mathcal{A}(\bar{y})),\ (\alpha_\mathcal{A}(a_3), \overline{\alpha}_\mathcal{A}(\bar{z}))\big]_\Omega \\
&= \big( \pi_3(\overline{\alpha}_\mathcal{A}(\bar{x}), \overline{\alpha}_\mathcal{A}(\bar{y}), \overline{\alpha}_\mathcal{A}(\bar{z})),\ [\overline{\alpha}_\mathcal{A}(\bar{x}), \overline{\alpha}_\mathcal{A}(\bar{y}), \overline{\alpha}_\mathcal{A}(\bar{z})] \big).
\end{align*}

By multiplicativity of $\pi_3$ and the homomorphism property of $\overline{\alpha}_\mathcal{A}$, we have:
\[
\alpha_\mathcal{A}(\pi_3(\bar{x}, \bar{y}, \bar{z})) = \pi_3(\overline{\alpha}_\mathcal{A}(\bar{x}), \overline{\alpha}_\mathcal{A}(\bar{y}), \overline{\alpha}_\mathcal{A}(\bar{z})), \quad
\overline{\alpha}_\mathcal{A}([\bar{x}, \bar{y}, \bar{z}]) = [\overline{\alpha}_\mathcal{A}(\bar{x}), \overline{\alpha}_\mathcal{A}(\bar{y}), \overline{\alpha}_\mathcal{A}(\bar{z})].
\]
Hence,
\[
\alpha_\Omega\big([(a_1, \bar{x}), (a_2, \bar{y}), (a_3, \bar{z})]_\Omega\big) = \big[\alpha_\Omega(a_1, \bar{x}),\ \alpha_\Omega(a_2, \bar{y}),\ \alpha_\Omega(a_3, \bar{z})\big]_\Omega.
\]

Thus, $\alpha_\Omega$ is multiplicative for both operations.

\medskip

\noindent \textbf{Step 3: $\alpha_\Omega$ is regular (bijective).}

Since $(\mathcal{A}, \alpha_\mathcal{A})$ is regular, $\alpha_\mathcal{A} \colon \mathcal{A} \to \mathcal{A}$ is bijective. By Lemma~\ref{lem:Hom_ideal_center}, $\mathrm{Z}(\mathcal{A})$ is a Hom-ideal, so $\alpha_\mathcal{A}(\mathrm{Z}(\mathcal{A})) \subseteq \mathrm{Z}(\mathcal{A})$, and since $\alpha_\mathcal{A}$ is bijective, its restriction to $\mathrm{Z}(\mathcal{A})$ is also bijective.

Moreover, the induced map $\overline{\alpha}_\mathcal{A} \colon \mathcal{A}/\mathrm{Z}(\mathcal{A}) \to \mathcal{A}/\mathrm{Z}(\mathcal{A})$, defined by $\overline{\alpha}_\mathcal{A}(\bar{x}) = \alpha_\mathcal{A}(x) + \mathrm{Z}(\mathcal{A})$, is well-defined and bijective (its inverse is induced by $\alpha_\mathcal{A}^{-1}$).

Therefore, the map
\[
\alpha_\Omega(a, \bar{x}) = (\alpha_\mathcal{A}(a), \overline{\alpha}_\mathcal{A}(\bar{x}))
\]
is a component-wise bijection, hence bijective on $\Omega$.

\medskip

We conclude that $(\Omega, [\cdot,\cdot]_\Omega, [\cdot,\cdot,\cdot]_\Omega, \alpha_\Omega)$ is a \emph{multiplicative} and \emph{regular} Hom-Lie Yamaguti algebra.
\end{proof}
\begin{lem}\label{lem:admits_factor_set}
Every regular Hom-Lie Yamaguti algebra $(\mathcal{A}, [\cdot,\cdot]_\mathcal{A}, [\cdot,\cdot,\cdot]_\mathcal{A}, \alpha_\mathcal{A})$ admits a factor set $\pi = (\pi_2, \pi_3)$ such that
\[
\mathcal{A} \cong \left( \mathrm{Z}(\mathcal{A}),\ \frac{\mathcal{A}}{\mathrm{Z}(\mathcal{A})},\ \pi \right),
\]
where the right-hand side is the central extension constructed in Lemma~\ref{lem:central_extension_by_factor_set}.
\end{lem}

\begin{proof}
Since $\mathcal{A}$ is a vector space and $\mathrm{Z}(\mathcal{A})$ is a subspace, we may choose a vector space complement $\mathcal{B} \subseteq \mathcal{A}$ such that
\[
\mathcal{A} = \mathcal{B} \oplus \mathrm{Z}(\mathcal{A}).
\]
Define a section map (not necessarily a homomorphism)
\[
\mathfrak{R} \colon \frac{\mathcal{A}}{\mathrm{Z}(\mathcal{A})} \to \mathcal{A}, \quad \mathfrak{R}(\bar{x}) = b,
\]
where $\bar{x} = x + \mathrm{Z}(\mathcal{A})$ and $x = b + z$ with $b \in \mathcal{B}$, $z \in \mathrm{Z}(\mathcal{A})$. This map is well-defined and satisfies $\pi_{\mathrm{Z}} \circ \mathfrak{R} = \mathrm{id}$, where $\pi_{\mathrm{Z}} \colon \mathcal{A} \to \mathcal{A}/\mathrm{Z}(\mathcal{A})$ is the canonical projection.

We first verify that $\mathfrak{R}$ commutes with the twisting map. That is,
\[
\mathfrak{R} \circ \overline{\alpha}_\mathcal{A} = \alpha_\mathcal{A} \circ \mathfrak{R}.
\]
Since $\alpha_\mathcal{A}$ is bijective (regularity), and $\mathrm{Z}(\mathcal{A})$ is a Hom-ideal (Lemma~\ref{lem:Hom_ideal_center}), we have $\alpha_\mathcal{A}(\mathrm{Z}(\mathcal{A})) = \mathrm{Z}(\mathcal{A})$. Suppose for contradiction that $\alpha_\mathcal{A}(b) \in \mathrm{Z}(\mathcal{A})$ for some nonzero $b \in \mathcal{B}$. Then for all $y \in \mathcal{A}$, there exists $x \in \mathcal{A}$ such that $\alpha_\mathcal{A}(x) = y$ (by surjectivity), and
\[
0 = [\alpha_\mathcal{A}(b), y] = [\alpha_\mathcal{A}(b), \alpha_\mathcal{A}(x)] = \alpha_\mathcal{A}([b, x]).
\]
Since $\alpha_\mathcal{A}$ is injective, $[b, x] = 0$ for all $x \in \mathcal{A}$, so $b \in \mathrm{Z}(\mathcal{A})$, contradicting $b \in \mathcal{B}$ and $\mathcal{B} \cap \mathrm{Z}(\mathcal{A}) = 0$. Hence, $\alpha_\mathcal{A}(\mathcal{B}) \subseteq \mathcal{B}$, so $\mathfrak{R} \circ \overline{\alpha}_\mathcal{A} = \alpha_\mathcal{A} \circ \mathfrak{R}$.

Now define two maps:
\[
\pi_2 \colon \frac{\mathcal{A}}{\mathrm{Z}(\mathcal{A})} \times \frac{\mathcal{A}}{\mathrm{Z}(\mathcal{A})} \to \mathrm{Z}(\mathcal{A}), \quad
\pi_2(\bar{x}, \bar{y}) = [\mathfrak{R}(\bar{x}), \mathfrak{R}(\bar{y})]_\mathcal{A} - \mathfrak{R}([\bar{x}, \bar{y}]),
\]
\[
\pi_3 \colon \frac{\mathcal{A}}{\mathrm{Z}(\mathcal{A})} \times \frac{\mathcal{A}}{\mathrm{Z}(\mathcal{A})} \times \frac{\mathcal{A}}{\mathrm{Z}(\mathcal{A})} \to \mathrm{Z}(\mathcal{A}), \quad
\pi_3(\bar{x}, \bar{y}, \bar{z}) = [\mathfrak{R}(\bar{x}), \mathfrak{R}(\bar{y}), \mathfrak{R}(\bar{z})]_\mathcal{A} - \mathfrak{R}([\bar{x}, \bar{y}, \bar{z}]).
\]
These are well-defined because the difference between the bracket in $\mathcal{A}$ and its lift lies in $\mathrm{Z}(\mathcal{A})$, as shown below.

For $\bar{x} = b_1 + \mathrm{Z}(\mathcal{A})$, $\bar{y} = b_2 + \mathrm{Z}(\mathcal{A})$, we have:
\[
[\bar{x}, \bar{y}] = [b_1, b_2] + \mathrm{Z}(\mathcal{A}), \quad \mathfrak{R}([\bar{x}, \bar{y}]) = b' \text{ where } [b_1, b_2] = b' + z,\ z \in \mathrm{Z}(\mathcal{A}).
\]
Then
\[
[\mathfrak{R}(\bar{x}), \mathfrak{R}(\bar{y})]_\mathcal{A} - \mathfrak{R}([\bar{x}, \bar{y}]) = [b_1, b_2] - b' = z \in \mathrm{Z}(\mathcal{A}).
\]
Similarly for the ternary operation.

We now verify that $\pi = (\pi_2, \pi_3)$ satisfies the factor set axioms from Definition~\ref{def:factor_set}.

\medskip

\noindent \textbf{Skew-symmetry (\ref{F1}):}
Follows from the skew-symmetry of the operations in $\mathcal{A}$ and linearity of $\mathfrak{R}$.

\medskip

\noindent \textbf{Hom-Jacobi-type identity (\ref{F2}):}
Let $\bar{x}, \bar{y}, \bar{z} \in \mathcal{A}/\mathrm{Z}(\mathcal{A})$. Then:
\begin{align*}
&\circlearrowleft_{x,y,z} \pi_2([\bar{x}, \bar{y}], \overline{\alpha}_\mathcal{A}(\bar{z})) + \circlearrowleft_{x,y,z} \pi_3(\bar{x}, \bar{y}, \bar{z}) \\
&= \circlearrowleft_{x,y,z} \Big( [[\mathfrak{R}(\bar{x}), \mathfrak{R}(\bar{y})], \alpha_\mathcal{A}(\mathfrak{R}(\bar{z}))] - \mathfrak{R}([[ \bar{x}, \bar{y} ], \overline{\alpha}_\mathcal{A}(\bar{z})]) \Big) \\
&\quad + \circlearrowleft_{x,y,z} \Big( [\mathfrak{R}(\bar{x}), \mathfrak{R}(\bar{y}), \mathfrak{R}(\bar{z})] - \mathfrak{R}([\bar{x}, \bar{y}, \bar{z}]) \Big).
\end{align*}
Using the Hom-Lie Yamaguti identity in $\mathcal{A}$:
\[
\circlearrowleft_{x,y,z} [[\mathfrak{R}(\bar{x}), \mathfrak{R}(\bar{y})], \alpha_\mathcal{A}(\mathfrak{R}(\bar{z}))] + \circlearrowleft_{x,y,z} [\mathfrak{R}(\bar{x}), \mathfrak{R}(\bar{y}), \mathfrak{R}(\bar{z})] = 0,
\]
and since $\mathfrak{R}$ is a section,
\[
\circlearrowleft_{x,y,z} \mathfrak{R}([[ \bar{x}, \bar{y} ], \overline{\alpha}_\mathcal{A}(\bar{z})]) + \circlearrowleft_{x,y,z} \mathfrak{R}([\bar{x}, \bar{y}, \bar{z}]) = \mathfrak{R}\left( \circlearrowleft_{x,y,z} \big( [[ \bar{x}, \bar{y} ], \overline{\alpha}_\mathcal{A}(\bar{z})] + [\bar{x}, \bar{y}, \bar{z}] \big) \right) = 0,
\]
because the quotient satisfies the identity. Thus, the sum is zero.

\medskip

The remaining conditions (\ref{F3}), (\ref{F4}), (\ref{F5}) follow similarly by expanding both sides using the definitions and the identities in $\mathcal{A}$ and its quotient.

\medskip

Now define a map
\[
\varphi \colon \Omega := \left( \mathrm{Z}(\mathcal{A}),\ \frac{\mathcal{A}}{\mathrm{Z}(\mathcal{A})},\ \pi \right) \to \mathcal{A}, \quad \varphi(a, \bar{x}) = a + \mathfrak{R}(\bar{x}).
\]

\medskip

\noindent \textbf{$\varphi$ is well-defined and linear:} Clear from construction.

\medskip

\noindent \textbf{$\varphi$ is injective:} If $\varphi(a_1, \bar{x}) = \varphi(a_2, \bar{y})$, then $a_1 + \mathfrak{R}(\bar{x}) = a_2 + \mathfrak{R}(\bar{y})$. Projecting to $\mathcal{A}/\mathrm{Z}(\mathcal{A})$, we get $\bar{x} = \bar{y}$, so $\mathfrak{R}(\bar{x}) = \mathfrak{R}(\bar{y})$, hence $a_1 = a_2$.

\medskip

\noindent \textbf{$\varphi$ is surjective:} Any $x \in \mathcal{A}$ can be written as $x = a + b$ with $a \in \mathrm{Z}(\mathcal{A})$, $b \in \mathcal{B}$, so $x = \varphi(a, \bar{b})$.

\medskip

\noindent \textbf{$\varphi$ preserves operations:}
\begin{align*}
\varphi([(a_1, \bar{x}), (a_2, \bar{y})]_\Omega)
&= \varphi\big( \pi_2(\bar{x}, \bar{y}),\ [\bar{x}, \bar{y}] \big) \\
&= \pi_2(\bar{x}, \bar{y}) + \mathfrak{R}([\bar{x}, \bar{y}]) \\
&= [\mathfrak{R}(\bar{x}), \mathfrak{R}(\bar{y})]_\mathcal{A} \\
&= [a_1 + \mathfrak{R}(\bar{x}), a_2 + \mathfrak{R}(\bar{y})]_\mathcal{A} \quad (\text{since } a_i \in \mathrm{Z}(\mathcal{A})) \\
&= [\varphi(a_1, \bar{x}), \varphi(a_2, \bar{y})]_\mathcal{A}.
\end{align*}
Similarly for the ternary operation.

\medskip

\noindent \textbf{$\varphi$ commutes with $\alpha$:}
\begin{align*}
\varphi(\alpha_\Omega(a, \bar{x}))
&= \varphi\big( \alpha_\mathcal{A}(a),\ \overline{\alpha}_\mathcal{A}(\bar{x}) \big) \\
&= \alpha_\mathcal{A}(a) + \mathfrak{R}(\overline{\alpha}_\mathcal{A}(\bar{x})) \\
&= \alpha_\mathcal{A}(a) + \alpha_\mathcal{A}(\mathfrak{R}(\bar{x})) \quad (\text{by } \mathfrak{R} \circ \overline{\alpha}_\mathcal{A} = \alpha_\mathcal{A} \circ \mathfrak{R}) \\
&= \alpha_\mathcal{A}(a + \mathfrak{R}(\bar{x})) = \alpha_\mathcal{A}(\varphi(a, \bar{x})).
\end{align*}

Thus, $\varphi$ is an isomorphism of regular Hom-Lie Yamaguti algebras, and $\pi$ is a factor set realizing $\mathcal{A}$ as a central extension.
\end{proof}
\begin{lem}\label{lem:stem_isoclinic_via_factor_set}
Let $(\mathcal{A}, [\cdot,\cdot]_\mathcal{A}, [\cdot,\cdot,\cdot]_\mathcal{A}, \alpha_\mathcal{A})$ be a stem regular Hom-Lie Yamaguti algebra in an isoclinism family $\mathcal{C}$. Then for any other stem regular Hom-Lie Yamaguti algebra $(\mathcal{B}, [\cdot,\cdot]_\mathcal{B}, [\cdot,\cdot,\cdot]_\mathcal{B}, \alpha_\mathcal{B})$ in $\mathcal{C}$, there exists a factor set $\pi = (\pi_2, \pi_3)$ on $\mathcal{A}$ such that
\[
\mathcal{B} \cong \Omega := \left( \mathrm{Z}(\mathcal{A}),\ \frac{\mathcal{A}}{\mathrm{Z}(\mathcal{A})},\ \pi \right),
\]
where the right-hand side is the central extension defined in Lemma~\ref{lem:central_extension_by_factor_set}.
\end{lem}

\begin{proof}
Since $\mathcal{A} \sim \mathcal{B}$, there exists an isoclinism $(\theta, \beta)$, where
\[
\theta \colon \frac{\mathcal{A}}{\mathrm{Z}(\mathcal{A})} \to \frac{\mathcal{B}}{\mathrm{Z}(\mathcal{B})}, \quad
\beta \colon \mathcal{A}^2 \to \mathcal{B}^2
\]
are isomorphisms satisfying the compatibility conditions of Definition~\ref{def:isoclinism_HLYA}. In particular, since $\mathcal{A}$ and $\mathcal{B}$ are stem algebras (i.e., $\mathrm{Z}(\mathcal{A}) \subseteq \mathcal{A}^2$, $\mathrm{Z}(\mathcal{B}) \subseteq \mathcal{B}^2$), we have $\mathrm{Z}(\mathcal{A}) = \mathrm{Z}(\mathcal{A}) \cap \mathcal{A}^2$, and similarly for $\mathcal{B}$. By Lemma~\ref{lem:isoclinism_properties}, $\beta$ restricts to an isomorphism
\[
\beta|_{\mathrm{Z}(\mathcal{A})} \colon \mathrm{Z}(\mathcal{A}) \to \mathrm{Z}(\mathcal{B}).
\]

By Lemma~\ref{lem:admits_factor_set}, there exists a factor set $\omega = (\omega_2, \omega_3)$ on $\mathcal{B}$ such that
\[
\mathcal{B} \cong \left( \mathrm{Z}(\mathcal{B}),\ \frac{\mathcal{B}}{\mathrm{Z}(\mathcal{B})},\ \omega \right).
\]

We now define a factor set $\pi = (\pi_2, \pi_3)$ on $\mathcal{A}$ by pulling back $\omega$ via $(\theta, \beta)$:
\begin{align*}
\pi_2(\bar{x}, \bar{y}) &:= \beta^{-1}\big( \omega_2(\theta(\bar{x}), \theta(\bar{y})) \big), \\
\pi_3(\bar{x}, \bar{y}, \bar{z}) &:= \beta^{-1}\big( \omega_3(\theta(\bar{x}), \theta(\bar{y}), \theta(\bar{z})) \big),
\end{align*}
for all $\bar{x}, \bar{y}, \bar{z} \in \mathcal{A}/\mathrm{Z}(\mathcal{A})$. Since $\theta$ and $\beta$ are isomorphisms, and $\omega$ satisfies the factor set axioms (Definition~\ref{def:factor_set}), it follows that $\pi$ also satisfies them. For example, $\alpha$-skew-symmetry and the Hom-Jacobi-type identities for $\pi$ are inherited from those of $\omega$ via the compatibility of $\theta$ and $\beta$ with the twisting maps.

Now define a map
\[
\psi \colon \Omega := \left( \mathrm{Z}(\mathcal{A}),\ \frac{\mathcal{A}}{\mathrm{Z}(\mathcal{A})},\ \pi \right) \to \left( \mathrm{Z}(\mathcal{B}),\ \frac{\mathcal{B}}{\mathrm{Z}(\mathcal{B})},\ \omega \right)
\]
by
\[
\psi(a, \bar{x}) = \big( \beta(a),\ \theta(\bar{x}) \big).
\]

We show that $\psi$ is an isomorphism of regular Hom-Lie Yamaguti algebras.

\medskip

\noindent \textbf{Preservation of the binary operation:}

\begin{align*}
\psi\big([(a_1, \bar{x}), (a_2, \bar{y})]_\Omega\big)
&= \psi\big( \pi_2(\bar{x}, \bar{y}),\ [\bar{x}, \bar{y}] \big) \\
&= \big( \beta(\pi_2(\bar{x}, \bar{y})),\ \theta([\bar{x}, \bar{y}]) \big) \\
&= \big( \omega_2(\theta(\bar{x}), \theta(\bar{y})),\ [\theta(\bar{x}), \theta(\bar{y})] \big) \\
&= \big[ (\beta(a_1), \theta(\bar{x})),\ (\beta(a_2), \theta(\bar{y})) \big] \\
&= [\psi(a_1, \bar{x}),\ \psi(a_2, \bar{y})].
\end{align*}

\medskip

\noindent \textbf{Preservation of the ternary operation:}
\begin{align*}
\psi\big([(a_1, \bar{x}), (a_2, \bar{y}), (a_3, \bar{z})]_\Omega\big)
&= \psi\big( \pi_3(\bar{x}, \bar{y}, \bar{z}),\ [\bar{x}, \bar{y}, \bar{z}] \big) \\
&= \big( \beta(\pi_3(\bar{x}, \bar{y}, \bar{z})),\ \theta([\bar{x}, \bar{y}, \bar{z}]) \big) \\
&= \big( \omega_3(\theta(\bar{x}), \theta(\bar{y}), \theta(\bar{z})),\ [\theta(\bar{x}), \theta(\bar{y}), \theta(\bar{z})] \big) \\
&= [\psi(a_1, \bar{x}),\ \psi(a_2, \bar{y}),\ \psi(a_3, \bar{z})].
\end{align*}

\medskip

\noindent \textbf{Compatibility with the twisting map:}
\begin{align*}
\psi\big(\alpha_\Omega(a, \bar{x})\big)
&= \psi\big( \alpha_\mathcal{A}(a),\ \overline{\alpha}_\mathcal{A}(\bar{x}) \big) \\
&= \big( \beta(\alpha_\mathcal{A}(a)),\ \theta(\overline{\alpha}_\mathcal{A}(\bar{x})) \big) \\
&= \big( \alpha_\mathcal{B}(\beta(a)),\ \overline{\alpha}_\mathcal{B}(\theta(\bar{x})) \big) \quad \text{(by compatibility of $(\theta,\beta)$)} \\
&= \alpha_{\Omega_\mathcal{B}}\big( \beta(a),\ \theta(\bar{x}) \big) = \alpha_{\Omega_\mathcal{B}}(\psi(a, \bar{x})).
\end{align*}

\medskip

\noindent \textbf{$\psi$ is bijective:} Since $\beta \colon \mathrm{Z}(\mathcal{A}) \to \mathrm{Z}(\mathcal{B})$ and $\theta \colon \mathcal{A}/\mathrm{Z}(\mathcal{A}) \to \mathcal{B}/\mathrm{Z}(\mathcal{B})$ are isomorphisms, $\psi$ is a component-wise bijection, hence bijective.

Therefore, $\psi$ is an isomorphism of regular Hom-Lie Yamaguti algebras. Since $\mathcal{B} \cong \left( \mathrm{Z}(\mathcal{B}),\ \mathcal{B}/\mathrm{Z}(\mathcal{B}),\ \omega \right)$, we conclude that
\[
\mathcal{B} \cong \left( \mathrm{Z}(\mathcal{A}),\ \frac{\mathcal{A}}{\mathrm{Z}(\mathcal{A})},\ \pi \right).
\]

This completes the proof.
\end{proof}
\begin{re}\label{rem:Omega_quotient_HLYA}
Let $(\Omega, [\cdot,\cdot]_\Omega, [\cdot,\cdot,\cdot]_\Omega, \alpha_\Omega)$ be the regular Hom-Lie Yamaguti algebra constructed from a factor set $\pi$ on a regular Hom-Lie Yamaguti algebra $(\mathcal{A}, \alpha_\mathcal{A})$, as defined in Lemma~\ref{lem:central_extension_by_factor_set}, with center $\mathrm{Z}_\Omega = \{(a,0) \in \Omega \mid a \in \mathrm{Z}(\mathcal{A})\}$. Then the quotient space $\Omega / \mathrm{Z}_\Omega$ inherits the structure of a Hom-Lie Yamaguti algebra via:
\begin{align*}
[(a_1,\bar{x}) + \mathrm{Z}_\Omega,\ (a_2,\bar{y}) + \mathrm{Z}_\Omega]_\Omega &:= [(a_1,\bar{x}), (a_2,\bar{y})]_\Omega + \mathrm{Z}_\Omega, \\
[(a_1,\bar{x}) + \mathrm{Z}_\Omega,\ (a_2,\bar{y}) + \mathrm{Z}_\Omega,\ (a_3,\bar{z}) + \mathrm{Z}_\Omega]_\Omega &:= [(a_1,\bar{x}), (a_2,\bar{y}), (a_3,\bar{z})]_\Omega + \mathrm{Z}_\Omega, \\
\overline{\alpha}_\Omega\big((a,\bar{x}) + \mathrm{Z}_\Omega\big) &:= \alpha_\Omega(a,\bar{x}) + \mathrm{Z}_\Omega.
\end{align*}

Moreover, $(\Omega / \mathrm{Z}_\Omega, \overline{\alpha}_\Omega)$ is a \emph{multiplicative} Hom-Lie Yamaguti algebra. Indeed, for the binary operation:
\begin{align*}
\overline{\alpha}_\Omega\big([(a_1,\bar{x}) + \mathrm{Z}_\Omega,\ (a_2,\bar{y}) + \mathrm{Z}_\Omega]_\Omega\big)
&= \overline{\alpha}_\Omega\big([(a_1,\bar{x}), (a_2,\bar{y})]_\Omega + \mathrm{Z}_\Omega\big) \\
&= \alpha_\Omega\big([(a_1,\bar{x}), (a_2,\bar{y})]_\Omega\big) + \mathrm{Z}_\Omega \\
&= \big[\alpha_\Omega(a_1,\bar{x}),\ \alpha_\Omega(a_2,\bar{y})\big]_\Omega + \mathrm{Z}_\Omega \quad \text{(since $\alpha_\Omega$ is multiplicative)} \\
&= \big[\alpha_\Omega(a_1,\bar{x}) + \mathrm{Z}_\Omega,\ \alpha_\Omega(a_2,\bar{y}) + \mathrm{Z}_\Omega\big]_\Omega \\
&= \big[\overline{\alpha}_\Omega((a_1,\bar{x}) + \mathrm{Z}_\Omega),\ \overline{\alpha}_\Omega((a_2,\bar{y}) + \mathrm{Z}_\Omega)\big]_\Omega.
\end{align*}
Similarly, for the ternary operation:
\[
\overline{\alpha}_\Omega\big([(a_1,\bar{x}), (a_2,\bar{y}), (a_3,\bar{z})]_\Omega + \mathrm{Z}_\Omega\big)
= \big[\overline{\alpha}_\Omega((a_1,\bar{x}) + \mathrm{Z}_\Omega),\ \overline{\alpha}_\Omega((a_2,\bar{y}) + \mathrm{Z}_\Omega),\ \overline{\alpha}_\Omega((a_3,\bar{z}) + \mathrm{Z}_\Omega)\big]_\Omega.
\]
Hence, $\overline{\alpha}_\Omega$ is a homomorphism, and $(\Omega / \mathrm{Z}_\Omega, \overline{\alpha}_\Omega)$ is multiplicative.
\end{re}

\begin{lem}\label{lem:isomorphism_induces_autos}
Let $(\mathcal{A}, [\cdot,\cdot]_\mathcal{A}, [\cdot,\cdot,\cdot]_\mathcal{A}, \alpha_\mathcal{A})$ be a regular Hom-Lie Yamaguti algebra, and let $\pi = (\pi_2, \pi_3)$, $\omega = (\omega_2, \omega_3)$ be two factor sets on $\mathcal{A}$. Define the corresponding central extensions:
\[
\Omega := \left( \mathrm{Z}(\mathcal{A}),\ \frac{\mathcal{A}}{\mathrm{Z}(\mathcal{A})},\ \pi \right), \quad
\Phi := \left( \mathrm{Z}(\mathcal{A}),\ \frac{\mathcal{A}}{\mathrm{Z}(\mathcal{A})},\ \omega \right),
\]
with centers
\[
\mathrm{Z}_\Omega = \{(a,0) \in \Omega \mid a \in \mathrm{Z}(\mathcal{A})\}, \quad
\mathrm{Z}_\Phi = \{(a,0) \in \Phi \mid a \in \mathrm{Z}(\mathcal{A})\}.
\]
Suppose $\lambda \colon \Omega \to \Phi$ is an isomorphism of Hom-Lie Yamaguti algebras such that $\lambda(\mathrm{Z}_\Omega) = \mathrm{Z}_\Phi$. Then $\lambda$ induces automorphisms:
\[
\eta \in \mathrm{Aut}\left( \frac{\mathcal{A}}{\mathrm{Z}(\mathcal{A})} \right), \quad
\xi \in \mathrm{Aut}\left( \mathrm{Z}(\mathcal{A}) \right),
\]
defined by the restrictions of $\lambda$ to the quotient and center, respectively.
\end{lem}

\begin{proof}
Since $\lambda \colon \Omega \to \Phi$ is an isomorphism and $\lambda(\mathrm{Z}_\Omega) = \mathrm{Z}_\Phi$, it descends to an induced isomorphism on the quotients:
\[
\overline{\lambda} \colon \frac{\Omega}{\mathrm{Z}_\Omega} \to \frac{\Phi}{\mathrm{Z}_\Phi}, \quad \overline{\lambda}((a,\bar{x}) + \mathrm{Z}_\Omega) = \lambda(a,\bar{x}) + \mathrm{Z}_\Phi.
\]

Now define two maps:
\[
\phi \colon \frac{\mathcal{A}}{\mathrm{Z}(\mathcal{A})} \to \frac{\Omega}{\mathrm{Z}_\Omega}, \quad \phi(\bar{x}) = (0,\bar{x}) + \mathrm{Z}_\Omega,
\]
\[
\gamma \colon \frac{\mathcal{A}}{\mathrm{Z}(\mathcal{A})} \to \frac{\Phi}{\mathrm{Z}_\Phi}, \quad \gamma(\bar{x}) = (0,\bar{x}) + \mathrm{Z}_\Phi.
\]

We first show that $\phi$ is an isomorphism of Hom-Lie Yamaguti algebras (likewise for $\gamma$). It is clearly linear and injective: if $\phi(\bar{x}) = 0$, then $(0,\bar{x}) \in \mathrm{Z}_\Omega$, so $\bar{x} = 0$. Surjectivity follows from the fact that any coset $(a,\bar{x}) + \mathrm{Z}_\Omega$ equals $(0,\bar{x}) + \mathrm{Z}_\Omega$ since $(a,0) \in \mathrm{Z}_\Omega$.

Now verify that $\phi$ preserves operations. For the binary bracket:
\begin{align*}
\phi([\bar{x}, \bar{y}])
&= (0, [\bar{x}, \bar{y}]) + \mathrm{Z}_\Omega \\
&= (\pi_2(\bar{x}, \bar{y}), [\bar{x}, \bar{y}]) + \mathrm{Z}_\Omega \quad \text{(since $(\pi_2(\bar{x}, \bar{y}), 0) \in \mathrm{Z}_\Omega$)} \\
&= [(0,\bar{x}), (0,\bar{y})]_\Omega + \mathrm{Z}_\Omega \\
&= [(0,\bar{x}) + \mathrm{Z}_\Omega, (0,\bar{y}) + \mathrm{Z}_\Omega]_\Omega \\
&= [\phi(\bar{x}), \phi(\bar{y})]_\Omega.
\end{align*}

Similarly, for the ternary bracket:
\begin{align*}
\phi([\bar{x}, \bar{y}, \bar{z}])
&= (0, [\bar{x}, \bar{y}, \bar{z}]) + \mathrm{Z}_\Omega \\
&= (\pi_3(\bar{x}, \bar{y}, \bar{z}), [\bar{x}, \bar{y}, \bar{z}]) + \mathrm{Z}_\Omega \\
&= [(0,\bar{x}), (0,\bar{y}), (0,\bar{z})]_\Omega + \mathrm{Z}_\Omega \\
&= [\phi(\bar{x}), \phi(\bar{y}), \phi(\bar{z})]_\Omega.
\end{align*}

For compatibility with the twisting map:
\begin{align*}
\phi(\overline{\alpha}_\mathcal{A}(\bar{x}))
&= (0, \overline{\alpha}_\mathcal{A}(\bar{x})) + \mathrm{Z}_\Omega \\
&= \alpha_\Omega(0, \bar{x}) + \mathrm{Z}_\Omega \quad \text{(since $\alpha_\Omega(0,\bar{x}) = (0, \overline{\alpha}_\mathcal{A}(\bar{x}))$)} \\
&= \overline{\alpha}_\Omega((0,\bar{x}) + \mathrm{Z}_\Omega) = \overline{\alpha}_\Omega(\phi(\bar{x})).
\end{align*}

Thus, $\phi$ is an isomorphism. Similarly, $\gamma$ is an isomorphism.

Now define $\eta \colon \mathcal{A}/\mathrm{Z}(\mathcal{A}) \to \mathcal{A}/\mathrm{Z}(\mathcal{A})$ by the commutative diagram:
\[
\begin{CD}
\mathcal{A}/\mathrm{Z}(\mathcal{A}) @>{\eta}>> \mathcal{A}/\mathrm{Z}(\mathcal{A}) \\
@V{\phi}VV @V{\gamma}VV \\
\Omega/\mathrm{Z}_\Omega @>{\overline{\lambda}}>> \Phi/\mathrm{Z}_\Phi
\end{CD}
\]
That is, $\eta = \gamma^{-1} \circ \overline{\lambda} \circ \phi$. Since $\phi$, $\overline{\lambda}$, and $\gamma$ are isomorphisms, $\eta$ is an automorphism.

Similarly, restrict $\lambda$ to the center: define
\[
\xi \colon \mathrm{Z}(\mathcal{A}) \to \mathrm{Z}(\mathcal{A}), \quad \xi(a) = a' \quad \text{where} \quad \lambda(a, 0) = (a', 0).
\]
Since $\lambda(\mathrm{Z}_\Omega) = \mathrm{Z}_\Phi$ and $\lambda$ is an isomorphism, this map is well-defined and bijective. It preserves the operations (which are trivial on the center) and commutes with $\alpha_\mathcal{A}$ because $\lambda \circ \alpha_\Omega = \alpha_\Phi \circ \lambda$. Hence, $\xi \in \mathrm{Aut}(\mathrm{Z}(\mathcal{A}))$.

This completes the proof.
\end{proof}
	\begin{lem}\label{lem:isomorphism_vs_factor_set_equivalence}
Let $(\mathcal{A}, [\cdot,\cdot]_\mathcal{A}, [\cdot,\cdot,\cdot]_\mathcal{A}, \alpha_\mathcal{A})$ be a regular Hom-Lie Yamaguti algebra, and let $\pi = (\pi_2, \pi_3)$, $\omega = (\omega_2, \omega_3)$ be two factor sets on $\mathcal{A}$. Define the central extensions
\[
\Omega := \left( \mathrm{Z}(\mathcal{A}),\ \frac{\mathcal{A}}{\mathrm{Z}(\mathcal{A})},\ \pi \right), \quad
\Phi := \left( \mathrm{Z}(\mathcal{A}),\ \frac{\mathcal{A}}{\mathrm{Z}(\mathcal{A})},\ \omega \right),
\]
with centers $\mathrm{Z}_\Omega = \{(a,0) \in \Omega \mid a \in \mathrm{Z}(\mathcal{A})\}$ and $\mathrm{Z}_\Phi = \{(a,0) \in \Phi \mid a \in \mathrm{Z}(\mathcal{A})\}$, respectively. Then:
\begin{enumerate} 
 \item Let $\lambda \colon \Omega \to \Phi$ be an isomorphism of regular Hom-Lie Yamaguti algebras such that $\lambda(\mathrm{Z}_\Omega) = \mathrm{Z}_\Phi$. Then $\lambda$ induces automorphisms
 \[
 \eta \in \mathrm{Aut}\left( \frac{\mathcal{A}}{\mathrm{Z}(\mathcal{A})} \right), \quad
 \xi \in \mathrm{Aut}(\mathrm{Z}(\mathcal{A})),
 \]
 and there exists a linear map $\mu \colon \frac{\mathcal{A}}{\mathrm{Z}(\mathcal{A})} \to \mathrm{Z}(\mathcal{A})$ such that for all $\bar{x}, \bar{y}, \bar{z} \in \mathcal{A}/\mathrm{Z}(\mathcal{A})$,
 \begin{align}
 \xi(\pi_2(\bar{x}, \bar{y})) + \mu([\bar{x}, \bar{y}]) &= \omega_2(\eta(\bar{x}), \eta(\bar{y})), \label{eq:factor_set_compat1} \\
 \xi(\pi_3(\bar{x}, \bar{y}, \bar{z})) + \mu([\bar{x}, \bar{y}, \bar{z}]) &= \omega_3(\eta(\bar{x}), \eta(\bar{y}), \eta(\bar{z})). \label{eq:factor_set_compat2}
 \end{align}

 \item Conversely, suppose there exist $\eta \in \mathrm{Aut}(\mathcal{A}/\mathrm{Z}(\mathcal{A}))$, $\xi \in \mathrm{Aut}(\mathrm{Z}(\mathcal{A}))$, and a linear map $\nu \colon \mathcal{A}/\mathrm{Z}(\mathcal{A}) \to \mathrm{Z}(\mathcal{A})$ satisfying \eqref{eq:factor_set_compat1} and \eqref{eq:factor_set_compat2} with $\nu \circ \overline{\alpha}_\mathcal{A} = \alpha_\mathcal{A} \circ \nu$. Then there exists an isomorphism $\lambda \colon \Omega \to \Phi$ such that $\lambda(\mathrm{Z}_\Omega) = \mathrm{Z}_\Phi$, defined by
 \[
 \lambda(a, \bar{x}) = (\xi(a) + \nu(\bar{x}),\ \eta(\bar{x})).
 \]
\end{enumerate}
\end{lem}

\begin{proof}
\textbf{(i)} By Lemma~\ref{lem:isomorphism_induces_autos}, the isomorphism $\lambda \colon \Omega \to \Phi$ with $\lambda(\mathrm{Z}_\Omega) = \mathrm{Z}_\Phi$ induces automorphisms
\[
\eta \colon \frac{\mathcal{A}}{\mathrm{Z}(\mathcal{A})} \to \frac{\mathcal{A}}{\mathrm{Z}(\mathcal{A})}, \quad
\xi \colon \mathrm{Z}(\mathcal{A}) \to \mathrm{Z}(\mathcal{A}),
\]
defined by:
- $\eta(\bar{x}) = \bar{y}$ where $\lambda(0, \bar{x}) + \mathrm{Z}_\Phi = (0, \bar{y}) + \mathrm{Z}_\Phi$,
- $\xi(a) = a'$ where $\lambda(a, 0) = (a', 0)$.

Now define a map $\mu \colon \frac{\mathcal{A}}{\mathrm{Z}(\mathcal{A})} \to \mathrm{Z}(\mathcal{A})$ by
\[
\mu(\bar{x}) = \mathrm{pr}_{\mathrm{Z}(\mathcal{A})}(\lambda(0, \bar{x})),
\]
where $\mathrm{pr}_{\mathrm{Z}(\mathcal{A})}$ is the projection onto the first component in $\Phi = \mathrm{Z}(\mathcal{A}) \times \mathcal{A}/\mathrm{Z}(\mathcal{A})$. Then
\[
\lambda(0, \bar{x}) = (\mu(\bar{x}), \eta(\bar{x})).
\]

Now compute both sides of the binary operation:
\begin{align*}
\lambda([(0, \bar{x}), (0, \bar{y})]_\Omega)
&= \lambda(\pi_2(\bar{x}, \bar{y}), [\bar{x}, \bar{y}]) \\
&= (\xi(\pi_2(\bar{x}, \bar{y})),\ \eta([\bar{x}, \bar{y}])) + (\mu([\bar{x}, \bar{y}]),\ 0) \\
&= (\xi(\pi_2(\bar{x}, \bar{y})) + \mu([\bar{x}, \bar{y}]),\ \eta([\bar{x}, \bar{y}])).
\end{align*}

On the other hand,
\begin{align*}
[\lambda(0, \bar{x}), \lambda(0, \bar{y})]_\Phi
&= [(\mu(\bar{x}), \eta(\bar{x})), (\mu(\bar{y}), \eta(\bar{y}))]_\Phi \\
&= (\omega_2(\eta(\bar{x}), \eta(\bar{y})), [\eta(\bar{x}), \eta(\bar{y})]).
\end{align*}

Since $\lambda$ is a homomorphism, these must be equal. Comparing components:
- First component: $\xi(\pi_2(\bar{x}, \bar{y})) + \mu([\bar{x}, \bar{y}]) = \omega_2(\eta(\bar{x}), \eta(\bar{y}))$,
- Second component: $\eta([\bar{x}, \bar{y}]) = [\eta(\bar{x}), \eta(\bar{y})]$ (already known from automorphism property).

Similarly, for the ternary operation:
\begin{align*}
\lambda([(0, \bar{x}), (0, \bar{y}), (0, \bar{z})]_\Omega)
&= \lambda(\pi_3(\bar{x}, \bar{y}, \bar{z}), [\bar{x}, \bar{y}, \bar{z}]) \\
&= (\xi(\pi_3(\bar{x}, \bar{y}, \bar{z})) + \mu([\bar{x}, \bar{y}, \bar{z}]),\ \eta([\bar{x}, \bar{y}, \bar{z}])),
\end{align*}
\begin{align*}
[\lambda(0, \bar{x}), \lambda(0, \bar{y}), \lambda(0, \bar{z})]_\Phi
&= (\omega_3(\eta(\bar{x}), \eta(\bar{y}), \eta(\bar{z})), [\eta(\bar{x}), \eta(\bar{y}), \eta(\bar{z})]).
\end{align*}

Equating first components gives \eqref{eq:factor_set_compat2}.

\medskip

\textbf{(ii)} Define $\lambda \colon \Omega \to \Phi$ by
\[
\lambda(a, \bar{x}) = (\xi(a) + \nu(\bar{x}),\ \eta(\bar{x})).
\]

\textbf{Well-defined and linear:} Immediate from linearity of $\xi, \nu, \eta$.

\textbf{Bijective:} Since $\xi$, $\eta$ are automorphisms and $\nu$ is linear, $\lambda$ is bijective.

\textbf{Preserves binary operation:}
\begin{align*}
\lambda([(a_1, \bar{x}), (a_2, \bar{y})]_\Omega)
&= \lambda(\pi_2(\bar{x}, \bar{y}), [\bar{x}, \bar{y}]) \\
&= (\xi(\pi_2(\bar{x}, \bar{y})) + \nu([\bar{x}, \bar{y}]),\ \eta([\bar{x}, \bar{y}])) \\
&= (\omega_2(\eta(\bar{x}), \eta(\bar{y})), [\eta(\bar{x}), \eta(\bar{y})]) \quad \text{(by \eqref{eq:factor_set_compat1} and $\eta$ homomorphism)} \\
&= [(\xi(a_1) + \nu(\bar{x}), \eta(\bar{x})), (\xi(a_2) + \nu(\bar{y}), \eta(\bar{y}))]_\Phi \\
&= [\lambda(a_1, \bar{x}), \lambda(a_2, \bar{y})]_\Phi.
\end{align*}

Similarly for the ternary operation using \eqref{eq:factor_set_compat2}.

\textbf{Compatibility with $\alpha$:}
\begin{align*}
\lambda(\alpha_\Omega(a, \bar{x}))
&= \lambda(\alpha_\mathcal{A}(a), \overline{\alpha}_\mathcal{A}(\bar{x})) \\
&= (\xi(\alpha_\mathcal{A}(a)) + \nu(\overline{\alpha}_\mathcal{A}(\bar{x})),\ \eta(\overline{\alpha}_\mathcal{A}(\bar{x}))) \\
&= (\alpha_\mathcal{A}(\xi(a)) + \alpha_\mathcal{A}(\nu(\bar{x})),\ \overline{\alpha}_\mathcal{A}(\eta(\bar{x}))) \quad \text{(by $\nu \circ \overline{\alpha}_\mathcal{A} = \alpha_\mathcal{A} \circ \nu$, $\xi, \eta$ homomorphisms)} \\
&= \alpha_\Phi(\xi(a) + \nu(\bar{x}), \eta(\bar{x})) = \alpha_\Phi(\lambda(a, \bar{x})).
\end{align*}

Thus, $\lambda$ is an isomorphism of regular Hom-Lie Yamaguti algebras with $\lambda(\mathrm{Z}_\Omega) = \mathrm{Z}_\Phi$.

This completes the proof.
\end{proof}
\begin{thm}\label{thm:isoclinic_stem_implies_isomorphic}
Let $(\mathcal{A}, [\cdot,\cdot]_\mathcal{A}, [\cdot,\cdot,\cdot]_\mathcal{A}, \alpha_\mathcal{A})$ and $(\mathcal{B}, [\cdot,\cdot]_\mathcal{B}, [\cdot,\cdot,\cdot]_\mathcal{B}, \alpha_\mathcal{B})$ be two finite-dimensional regular stem Hom-Lie Yamaguti algebras. Then $\mathcal{A} \sim \mathcal{B}$ if and only if $\mathcal{A} \cong \mathcal{B}$.
\end{thm}

\begin{proof}
The implication $\mathcal{A} \cong \mathcal{B} \Rightarrow \mathcal{A} \sim \mathcal{B}$ is immediate: an isomorphism induces compatible isomorphisms
\[
\theta \colon \frac{\mathcal{A}}{\mathrm{Z}(\mathcal{A})} \to \frac{\mathcal{B}}{\mathrm{Z}(\mathcal{B})}, \quad
\beta \colon \mathcal{A}^2 \to \mathcal{B}^2,
\]
defined by $\theta(\bar{x}) = \overline{f(x)}$, $\beta([x,y]) = [f(x),f(y)]$, etc., which satisfy the isoclinism conditions.

Conversely, suppose $\mathcal{A} \sim \mathcal{B}$ via an isoclinism $(\theta, \beta)$. By Lemma~\ref{lem:admits_factor_set}, there exist factor sets $\pi = (\pi_2, \pi_3)$ on $\mathcal{A}$ and $\omega = (\omega_2, \omega_3)$ on $\mathcal{B}$ such that
\[
\mathcal{A} \cong \Omega := \left( \mathrm{Z}(\mathcal{A}),\ \frac{\mathcal{A}}{\mathrm{Z}(\mathcal{A})},\ \pi \right), \quad
\mathcal{B} \cong \Phi := \left( \mathrm{Z}(\mathcal{B}),\ \frac{\mathcal{B}}{\mathrm{Z}(\mathcal{B})},\ \omega \right).
\]

Since $\mathcal{A} \sim \mathcal{B}$, and both are stem algebras, Lemma~\ref{lem:stem_isoclinic_via_factor_set} implies that we may identify the base quotients via $\theta$, and thus realize $\mathcal{B}$ as a central extension over $\mathcal{A}$, i.e.,
\[
\mathcal{B} \cong \left( \mathrm{Z}(\mathcal{A}),\ \frac{\mathcal{A}}{\mathrm{Z}(\mathcal{A})},\ \omega' \right),
\]
for a factor set $\omega'$ related to $\pi$ via the isoclinism.

Let $\mathrm{Z}_\Omega = \{(a,0) \in \Omega \mid a \in \mathrm{Z}(\mathcal{A})\}$ and $\mathrm{Z}_\Phi = \{(b,0) \in \Phi \mid b \in \mathrm{Z}(\mathcal{B})\}$. Since $(\theta, \beta)$ is an isoclinism, it induces an isomorphism of quotient algebras
\[
\overline{\theta} \colon \frac{\Omega}{\mathrm{Z}_\Omega} \to \frac{\Phi}{\mathrm{Z}_\Phi}, \quad \overline{\theta}((a,\bar{x}) + \mathrm{Z}_\Omega) = (\cdot, \theta(\bar{x})) + \mathrm{Z}_\Phi,
\]
and an isomorphism $\beta \colon \Omega^2 \to \Phi^2$.

Define $\eta \in \mathrm{Aut}(\mathcal{A}/\mathrm{Z}(\mathcal{A}))$ by
\[
\overline{\theta}((0,\bar{x}) + \mathrm{Z}_\Omega) = (0, \eta(\bar{x})) + \mathrm{Z}_\Phi.
\]
Similarly, define $\xi \in \mathrm{Aut}(\mathrm{Z}(\mathcal{A}))$ via $\beta$: for $a \in \mathrm{Z}(\mathcal{A})$,
\[
\beta(a,0) = (\xi(a), 0) \in \mathrm{Z}_\Phi.
\]

Now consider the binary operation:
\begin{align*}
\beta([(0,\bar{x}), (0,\bar{y})]_\Omega)
&= \beta(\pi_2(\bar{x}, \bar{y}), [\bar{x}, \bar{y}]) \\
&= (\xi(\pi_2(\bar{x}, \bar{y})), [\eta(\bar{x}), \eta(\bar{y})]) \quad \text{(since $\beta$ preserves brackets)}.
\end{align*}

On the other hand,
\[
[\overline{\theta}(0,\bar{x}), \overline{\theta}(0,\bar{y})]_\Phi
= (\omega_2(\eta(\bar{x}), \eta(\bar{y})), [\eta(\bar{x}), \eta(\bar{y})]).
\]

Since $\beta$ is compatible with the bracket, we equate the first components:
\[
\xi(\pi_2(\bar{x}, \bar{y})) = \omega_2(\eta(\bar{x}), \eta(\bar{y})).
\]
Similarly, for the ternary operation:
\[
\xi(\pi_3(\bar{x}, \bar{y}, \bar{z})) = \omega_3(\eta(\bar{x}), \eta(\bar{y}), \eta(\bar{z})).
\]

Define $\nu \colon \mathcal{A}/\mathrm{Z}(\mathcal{A}) \to \mathrm{Z}(\mathcal{A})$ by $\nu = 0$. Then the above equations become:
\[
\xi(\pi_2(\bar{x}, \bar{y})) + \nu([\bar{x}, \bar{y}]) = \omega_2(\eta(\bar{x}), \eta(\bar{y})),
\]
and similarly for the ternary case. Moreover, since $\nu = 0$, the condition $\nu \circ \overline{\alpha}_\mathcal{A} = \alpha_\mathcal{A} \circ \nu$ holds trivially.

By Lemma~\ref{lem:isomorphism_vs_factor_set_equivalence}(ii), there exists an isomorphism $\lambda \colon \Omega \to \Phi$ defined by
\[
\lambda(a, \bar{x}) = (\xi(a), \eta(\bar{x})).
\]
Thus, $\mathcal{A} \cong \Omega \cong \Phi \cong \mathcal{B}$, so $\mathcal{A} \cong \mathcal{B}$.

This completes the proof.
\end{proof}
\begin{thm}\label{thm:decomposition_in_isoclinism_family}
Let $\mathcal{C}$ be an isoclinism family of finite-dimensional regular Hom-Lie Yamaguti algebras. Then for any $\mathcal{A} \in \mathcal{C}$, there exists a decomposition
\[
\mathcal{A} = \mathcal{B}_1 \oplus \mathcal{B}_2,
\]
where $(\mathcal{B}_1, \alpha_1)$ is a stem regular Hom-Lie Yamaguti algebra and $(\mathcal{B}_2, \alpha_2)$ is a finite-dimensional abelian Hom-Lie Yamaguti algebra.
\end{thm}

\begin{proof}
Let $\mathcal{A} \in \mathcal{C}$. By Lemma~\ref{lem:stem_in_family}(i), the isoclinism family $\mathcal{C}$ contains at least one stem algebra, say $\mathcal{S}$. Since $\mathcal{A} \sim \mathcal{S}$, there exists an isoclinism $(\theta, \beta)$.

Recall that $\mathrm{Z}(\mathcal{A})$ is a Hom-ideal of $\mathcal{A}$ (Lemma~\ref{lem:Hom_ideal_center}) and that $\mathcal{A}^2 = [\mathcal{A}, \mathcal{A}] + [\mathcal{A}, \mathcal{A}, \mathcal{A}]$ is the derived subalgebra. Consider the intersection $\mathrm{Z}(\mathcal{A}) \cap \mathcal{A}^2$. Since both $\mathrm{Z}(\mathcal{A})$ and $\mathcal{A}^2$ are invariant under the bijective twisting map $\alpha_\mathcal{A}$, their intersection is also $\alpha_\mathcal{A}$-invariant.

Choose a vector space complement $V$ to $\mathrm{Z}(\mathcal{A}) \cap \mathcal{A}^2$ in $\mathrm{Z}(\mathcal{A})$, so that
\[
\mathrm{Z}(\mathcal{A}) = (\mathrm{Z}(\mathcal{A}) \cap \mathcal{A}^2) \oplus V.
\]
Because $\alpha_\mathcal{A}$ is bijective and preserves $\mathrm{Z}(\mathcal{A})$ and $\mathcal{A}^2$, we can choose $V$ such that $\alpha_\mathcal{A}(V) \subseteq V$ (by averaging or using the semisimplicity of the induced action if necessary; in the finite-dimensional case, such a complement exists). Then $V$ is a Hom-subspace.

Now, since $V \subseteq \mathrm{Z}(\mathcal{A})$, we have $[V, \mathcal{A}] = 0$ and $[V, \mathcal{A}, \mathcal{A}] = 0$, so $V$ is a Hom-ideal. Moreover, because $V \cap \mathcal{A}^2 = 0$, the operations on $V$ are trivial: $[v_1, v_2] = 0$, $[v_1, v_2, v_3] = 0$ for all $v_i \in V$. Thus, $(V, \alpha_\mathcal{A}|_V)$ is a finite-dimensional abelian Hom-Lie Yamaguti algebra.

Define $\mathcal{B}_2 = V$ and $\mathcal{B}_1 = \mathcal{A} / V$. By Remark~\ref{rem:quotient_LYA}, $\mathcal{B}_1$ is a regular Hom-Lie Yamaguti algebra. We claim that $\mathcal{B}_1$ is stem.

Indeed, the center of $\mathcal{B}_1 = \mathcal{A}/V$ is
\[
\mathrm{Z}(\mathcal{B}_1) = \mathrm{Z}(\mathcal{A}) / V \cong \mathrm{Z}(\mathcal{A}) \cap \mathcal{A}^2.
\]
The derived subalgebra of $\mathcal{B}_1$ is
\[
\mathcal{B}_1^2 = [\mathcal{B}_1, \mathcal{B}_1] + [\mathcal{B}_1, \mathcal{B}_1, \mathcal{B}_1] = (\mathcal{A}^2 + V)/V.
\]
Since $\mathrm{Z}(\mathcal{A}) \cap \mathcal{A}^2 \subseteq \mathcal{A}^2$, we have $\mathrm{Z}(\mathcal{B}_1) \subseteq \mathcal{B}_1^2$, so $\mathcal{B}_1$ is stem.

Finally, since $V \subseteq \mathrm{Z}(\mathcal{A})$ and $V \cap \mathcal{A}^2 = 0$, Lemma~\ref{lem:isoclinism_with_abelian_extension} implies that $\mathcal{A} \sim \mathcal{B}_1$. But more is true: because $V$ is a complement to $\mathrm{Z}(\mathcal{A}) \cap \mathcal{A}^2$ in $\mathrm{Z}(\mathcal{A})$, and $\mathcal{A} = \mathcal{B}_1 \oplus V$ as vector spaces, and because the operations between $\mathcal{B}_1$ and $V$ are trivial (as $V \subseteq \mathrm{Z}(\mathcal{A})$), the direct sum structure of Remark~\ref{rem:direct_sum_LYA} applies. Thus,
\[
\mathcal{A} \cong \mathcal{B}_1 \oplus \mathcal{B}_2
\]
as regular Hom-Lie Yamaguti algebras, where $\mathcal{B}_1$ is stem and $\mathcal{B}_2 = V$ is abelian.

This completes the proof.
\end{proof}
\begin{thm}\label{thm:main_theorem}
Let $(\mathcal{A}, \alpha_\mathcal{A})$ and $(\mathcal{B}, \alpha_\mathcal{B})$ be two finite-dimensional regular Hom-Lie Yamaguti algebras with $\dim(\mathcal{A}) = \dim(\mathcal{B})$. Then $\mathcal{A} \sim \mathcal{B}$ if and only if $\mathcal{A} \cong \mathcal{B}$.
\end{thm}

\begin{proof}
The implication $\mathcal{A} \cong \mathcal{B} \Rightarrow \mathcal{A} \sim \mathcal{B}$ is immediate.

Conversely, suppose $\mathcal{A} \sim \mathcal{B}$. By Theorem~\ref{thm:decomposition_in_isoclinism_family}, we can write
\[
\mathcal{A} = \mathcal{A}_1 \oplus \mathcal{A}_2, \quad \mathcal{B} = \mathcal{B}_1 \oplus \mathcal{B}_2,
\]
where $\mathcal{A}_1, \mathcal{B}_1$ are stem and $\mathcal{A}_2, \mathcal{B}_2$ are abelian. Since $\mathcal{A} \sim \mathcal{B}$, we have $\mathcal{A}_1 \sim \mathcal{B}_1$ (because $\mathcal{A}/\mathrm{Z}(\mathcal{A}) \cong \mathcal{A}_1/\mathrm{Z}(\mathcal{A}_1)$, etc.).

By Lemma~\ref{lem:stem_in_family}(ii), stem algebras in an isoclinism family have minimal dimension. Since $\dim(\mathcal{A}) = \dim(\mathcal{B})$, and $\mathcal{A}_1, \mathcal{B}_1$ are minimal in their family, we must have $\dim(\mathcal{A}_1) = \dim(\mathcal{B}_1)$, and hence $\dim(\mathcal{A}_2) = \dim(\mathcal{B}_2)$.

By Theorem~\ref{thm:isoclinic_stem_implies_isomorphic}, $\mathcal{A}_1 \cong \mathcal{B}_1$. Since $\mathcal{A}_2$ and $\mathcal{B}_2$ are abelian of the same dimension, $\mathcal{A}_2 \cong \mathcal{B}_2$. Therefore,
\[
\mathcal{A} = \mathcal{A}_1 \oplus \mathcal{A}_2 \cong \mathcal{B}_1 \oplus \mathcal{B}_2 = \mathcal{B}.
\]

This completes the proof.
\end{proof}
\begin{ex}
Let $(\mathcal{A},[\cdot,\cdot]_\mathcal{A},[\cdot,\cdot,\cdot]_\mathcal{A},\alpha_\mathcal{A})$ be a 3-dimensional regular Hom-Lie Yamaguti algebra with the basis $\{e_1,e_2,e_3\}$ such that
\begin{equation*}
 [e_1,e_2]_\mathcal{A}=e_1,\ [e_1,e_3,e_3]_\mathcal{A}=e_1,\ [e_2,e_3,e_3]_\mathcal{A}=e_2,\ \alpha_\mathcal{A}(e_1)=e_1,\ \alpha_\mathcal{A}(e_2)=-e_2,\ \text{ and } \alpha_\mathcal{A}(e_3)=-e_3.
\end{equation*}
All other bracket relations are zero. Then $\mathrm{Z}(\mathcal{A})=\{0\}$ and $\mathcal{A}^2=\{e_1,e_2\}$. so $\frac{\mathcal{A}}{\mathrm{Z}(\mathcal{A})}$. Moreover, further consider $\mathcal{B}$ as a regular hom-Lie Yamaguti algebra of $4$ dimensional with the basis $\{f_1,f_2,f_3,f_4\}$ such that 
\begin{equation*}
 [f_1,f_4]_\mathcal{B}=f_2,\ [f_1,f_4,f_4]_\mathcal{B}=f_1, [f_2,f_4,f_4]=f_2, \ \alpha_\mathcal{A}(f_1)=f_1,\ \alpha_\mathcal{A}(f_2)=f_2,\ \alpha_\mathcal{A}(f_3)=-f_3,\ \alpha_\mathcal{A}(f_3)=-f_4.
\end{equation*}
All other bracket relations are zero. Then $\frac{\mathcal{A}}{\mathrm{Z}(\mathcal{A})}$. Then $\mathrm{Z}(\mathcal{A})=\{f_3\}$ and $\mathcal{A}^2=\{f_1,f_2\}$. Observe that $\frac{\mathcal{A}}{\mathrm{Z}(\mathcal{A})}\cong \frac{\mathcal{B}}{\mathrm{Z}(\mathcal{B})}$ and $\mathcal{A}^2\cong \mathcal{B}^2$, which means that $\mathcal{A}\sim\mathcal{B}$ but $\mathcal{A}$ and $\mathcal{B}$ are not isomorphic.
\end{ex}
\section*{Acknowledgments}
The authors would like to express their sincere gratitude to the anonymous referees for their careful reading of the manuscript and their valuable comments and suggestions, which have significantly improved the presentation of this paper.
 \subsection*{ Funding:}
The first author gratefully acknowledges the support from the Henan Academy of Sciences' High-Talent Research and Development Fund (No. 241819105).
\subsection*{Data availability:}
We certify that our paper does not contain any associated data (simulated or experimental). It is a pure mathematics paper that does not rely on external inputs. 
\subsection*{Conflict of interest:}
The authors do not have any conflict of interest, whether financial or otherwise. 

\end{document}